\documentclass[11pt]{article}

\usepackage[a4paper,margin=1in]{geometry}

\usepackage{amsfonts} 
\usepackage{amsmath} 
\usepackage{amssymb} 
\usepackage{amsthm} 

\usepackage{enumerate} 
\usepackage{dutchcal} 

\usepackage{graphicx} 
\usepackage{xcolor} 
\usepackage{hyperref} 
\hypersetup{
    colorlinks=true,
    linkcolor=blue,
    filecolor=magenta,
    urlcolor=blue,
}
\usepackage{tikz}
\usetikzlibrary{matrix,arrows,positioning,decorations.markings,decorations.pathmorphing}

\usepackage{mathtools} 
\usepackage[activate={true,nocompatibility},final,tracking=true,kerning=true,spacing=true,factor=1100,stretch=10,shrink=10]{microtype}

\usepackage{authblk}  
\usepackage{orcidlink} 

\newtheorem{theorem}{Theorem}

\newtheorem{definition}[theorem]{Definition}

\newtheorem{lemma}[theorem]{Lemma}

\newtheorem{proposition}[theorem]{Proposition}
\newtheorem{remark}[theorem]{Remark}

\newcommand{\R}{\mathbb{R}}
\newcommand{\g}{\tilde{\mathfrak{g}}}
\newcommand{\kk}{\tilde{\mathfrak{k}}}

\newcommand{\A}{\mathcal{A}}
\newcommand{\vol}{\mathbf{v}}
\newcommand{\tp}{\tilde{p}}
\newcommand{\pr}{\mathrm{pr}}
\newcommand{\Hor}{\mathrm{Hor}}
\newcommand{\Ad}{\mathrm{Ad}}
\newcommand{\Curv}{\mathrm{Curv}}

\newcommand{\llangle}{\mathopen{\langle\!\langle}}
\newcommand{\rrangle}{\mathclose{\rangle\!\rangle}}

\title{Lie-Poisson reduction in principal bundles by a subgroup of the structure group}

\author[1]{M.\'A. Berbel \orcidlink{0000-0002-4269-9599}}  
\author[2]{M. Castrill\'on L\'opez \orcidlink{0000-0001-7673-9870}}

\affil[1]{Departamento de Matemática Aplicada, Universidad Pontificia Comillas,
Alberto Aguilera 25, 28015-Madrid, Spain\\
    \texttt{maberbel@comillas.edu}}
\affil[2]{Departamento de \'Algebra, Geometr\'\i a y Topolog\'\i a, Universidad Complutense de Madrid, Plaza de las Ciencias 3, 28040-Madrid, Spain\\
    \texttt{mcastri@mat.ucm.es}}
\date{}

\begin{document}

\maketitle
 
\begin{abstract}
We study Hamiltonian field theories on the multisymplectic bundle of a principal
$G$-bundle with Hamiltonian densities invariant under a subgroup $H \subset G$.
Using the covariant bracket formulation, we reduce the polysymplectic space and
derive the corresponding reduced observables, brackets, and equations of motion,
yielding a Lie--Poisson reduction by a subgroup for field theories. We also address
the reconstruction problem, characterizing reconstruction in terms of the flatness
of an associated connection. Several examples, including the heavy top, molecular
strands with broken symmetry, and affine principal bundles, illustrate the general
framework.
\end{abstract}
\vspace{4mm}

\noindent \em Keywords: \em Lie-Poisson, bracket, field theories, Hamiltonian, reduction.

\vspace{4mm}

\noindent \em Mathematics Subject Classification 2020: \em 70S05, 58D19, 70G45, 70S10



\section{Introduction}


Reduction in classical Field Theories with symmetry is a central tool for simplifying their constitutive equations as well as uncovering their underlying geometric structure, both in the Lagrangian and in the Hamiltonian frameworks. On the Hamiltonian side, which is the focus of this paper, several covariant approaches to reduction have been developed over the years. One line of research is based on the introduction of multimomentum maps, extending the Marsden–Weinstein reduction theorem to the multisymplectic setting (see, for instance, \cite{MultiGroupoids, multi_formalism, multivector_hamiltonian, RemarksMultiSymp, Madsen12, Madsen13}). A complementary approach, closer in spirit to Lie–Poisson reduction in Mechanics, relies on a covariant formulation of Poisson brackets (see \cite{GradedPoisson, AffineLieP, PoissonForms, CovariantPoisson}), allowing the dynamics to be expressed intrinsically in terms of reduced brackets and observables.
While these methods are well understood when the symmetry group coincides with the structure group of the underlying bundle, their extension to more general symmetry settings present additional geometric subtleties. More precisely, in \cite{someremarks} the covariant bracket formulation introduced in \cite{PoissonForms, CovariantPoisson} is reduced in the case of a field theory whose configuration bundle is a principal bundle $P \to M$ and the symmetry group equals the structure group $G$. For that, the polysymplectic bundle (see Definitions below) 
\begin{equation*}
    \Pi_P = TM \otimes V^*P \otimes \bigwedge\nolimits^n T^*M,
\end{equation*}
reduces to 
\begin{equation} \label{eq:intro-someremarks}
\Pi_P / G \cong TM \otimes \g^* \otimes \bigwedge\nolimits^n T^*M,
\end{equation}
where $\g$ denotes the adjoint bundle associated with $P$ and $n=\dim M$. This canonical identification leads to a covariant version of Lie--Poisson reduction for field theories, expressed entirely in terms of reduced brackets and observables.

More recently, this Hamiltonian reduction scheme was extended in \cite{PoissonPoincare} to arbitrary bundles $P\to M$ in which a Lie group $K$ is acting vertically and properly. In particular, if $P$ is a $G$-principal bundle and the Hamiltonian is invariant under a Lie subgroup $K \subset G$. In this setting, $P \to P/K$ is a principal bundle and, upon choosing a principal connection $\A$ on it, the reduced polysymplectic space admits a splitting of the form
\begin{equation} \label{eq:intro-poisson-poincare}
\Pi_P / K \cong_{\A}
\left( TM \otimes \kk^* \otimes \bigwedge\nolimits^n T^*M \right) \oplus \Pi_{P/K},
\end{equation}
where $\kk$ is the adjoint bundle of $P \to P/K$. However, the identification \eqref{eq:intro-poisson-poincare} is not canonical, as it depends explicitly on the choice of that auxiliary connection $\A$. This dependence motivates the search for a reduction procedure that avoids the introduction of additional, non-canonical structures.

The purpose of the present paper is to develop a Hamiltonian reduction procedure for field theories on principal bundles that are invariant under a subgroup of the structure group such that it is intrinsic and does not depend on the choice of a principal connection. Our approach extends the covariant Poisson reduction developed in \cite{someremarks} to the case of subgroup symmetry, leading to a description of the reduced polysymplectic space together with a reduced covariant bracket governing the dynamics. In this sense, the results obtained here may be viewed as a Hamiltonian counterpart of the Euler--Poincar\'e reduction by a subgroup introduced in \cite{EPsubgroup}, where the quotient of the first jet bundle $(J^1P)/K$ is canonically identified with a fibered product of the bundle of connections and the reduced configuration space,
\begin{equation} \label{eq:intro-poisson-subgroup}
(J^1P)/K \cong
\mathcal{C} \times_M P/K,
\end{equation}
yielding a zero--order variational problem subject constraints.

From the Hamiltonian perspective, we show that an analogous geometric structure emerges in the reduced polysymplectic space: It admits a natural splitting as the fibered product of a bundle related with the adjoint variables in \eqref{eq:intro-someremarks} and the reduced bundle $P/K$, without the introduction of any auxiliary connection. We further address the corresponding reconstruction problem and characterize the obstruction to reconstruction in terms of the curvature of certain connection. We note that nature of the compatibility conditions for reconstructions is a well-known fact in the Lagrangian side, but rather unclear in the Hamiltonian side . The analysis here solves the reconstruction problem on principal bundle in full generality,


The paper is organized as follows. In Section~\ref{sec: multi}, we review the main elements of Hamiltonian field theories, with particular emphasis on the covariant bracket formulation and the class of observables naturally defined on the polysymplectic space~$\Pi_{P}$. Section~\ref{sec: reduced polysymplectic} is devoted to the geometric structure of the reduced polysymplectic space~$\Pi_{P}/K$, where we analyze its properties and establish its identification with the fibered product $\Pi_{P}/G \times_{M} P/K$. The core results of the paper are presented in Section~\ref{sec: reduction}, which develops a Lie--Poisson reduction procedure by a subgroup. There, we introduce the reduced observables, define the reduced covariant brackets, and derive the corresponding reduced dynamics. Section~\ref{sec: reconstruction} addresses the reconstruction problem, identifying the flatness of a connection as the obstruction to reconstruction. Finally, Section~\ref{sec: examples} illustrates the theory through two relevant examples: $SO(3)$-molecular strands with broken symmetry and field theories defined on affine principal bundles.

\section{Multisymplectic and Polysymplectic bundles} \label{sec: multi}  

Given a fiber bundle $\pi_{M,P}:P\rightarrow M$, the \em dual jet bundle \em $J^1P^* $ is a vector bundle over $P$ whose fiber at a point $y \in P_x$, $x\in M$ is
\begin{equation*}
    J^1_yP^*=\mathrm{Aff}(J^1_yP,\bigwedge\nolimits^nT_x^*M),
\end{equation*}
i.e., the set of affine maps from the jet bundle to the bundle of $n$-forms on $M$, $n=\mathrm{dim} M$. Given a fiber coordinate system $(x^i,y^a)$ on $P$, we define adapted fiber coordinates $(x^i, y^a, p^i_a, \tp)$ on $J^1P^*$ as those such that affine maps read
\begin{equation}
y_i^a\mapsto (\tp+p^i_ay^a_i)d^nx,
\end{equation}
where $d^nx=dx^1\wedge\cdots\wedge dx^n$. Since $J^1P^*$ is defined as a bundle of affine morphisms, it follows that $\mathrm{rank}(J^1P^*)=\mathrm{rank}(J^1P)+1$. 

As is well known, the space $J^1P^*$ is canonically isomorphic to the subbundle \(Z \subset \bigwedge\nolimits^n T^*P\) consisting of those \(n\)-forms that vanish when contracted with any pair of vertical vectors. This identification endows $J^1P^* \cong Z$ with a \em canonical \em $n$-form $\Theta$, defined at each point \(z \in Z\) by 
\begin{equation}
\Theta(z)(u_1,\dots,u_n)=z\left(T\pi_{P,\bigwedge\nolimits^nT_x^*P}u_1,\dots,T\pi_{P,\bigwedge\nolimits^nT_x^*P}u_n\right)
\end{equation}
where $u_i\in T_zZ$ for $i=1,\dots,n$. The \em canonical multisymplectic form is the closed \em $(n+1)$-form
\begin{equation}
\Omega=-d\Theta.
\end{equation}
In local coordinates, these forms have the local expressions:
\begin{equation}
\Theta=p_a^idy^a\wedge d^{n-1}x_i+\tp d^nx,
\end{equation}
\begin{equation}
\Omega=dy^a\wedge dp_a^i \wedge d^{n-1}x_i-d\tp\wedge d^nx.
\end{equation}
A comprehensive treatment of the multisymplectic formalism can be found in \cite{multi_formalism, multisymplectic, gotay}. Note that the multisymplectic bundle defined in this paper is a multisymplectic manifold. 

The \emph{polysymplectic bundle} is defined as
\begin{equation}
\Pi_P := TM \otimes V^*P \otimes \bigwedge\nolimits ^n T^*M,
\end{equation}
which is a vector bundle over $P$ whose rank coincides with that of the first jet bundle $J^1P$.
This bundle may be identified with the dual of the vector bundle $T^*M \otimes VP$ modeling $J^1P$. More precisely, for each $y\in P$ with $\pi_{M,P}(y)=x$, the affine map
\[
\varphi \in J^1_yP^* := \mathrm{Aff}\!\left(J^1_yP,\bigwedge\nolimits^n T_x^*M\right)
\]
admits an associated linear map
\[
\vec{\varphi} \in (T_x^*M \otimes V_yP)^* \otimes \bigwedge\nolimits^n T_x^*M
\cong (\Pi_P)_y.
\]
The projection $\varphi \mapsto \vec{\varphi}$ equips $J^1P^*$ with the structure of an affine bundle over $\Pi_P$ modeled on the rank-one vector bundle $\bigwedge\nolimits^n T^*M \to M$ (see \cite{AffineBracket}).

A \emph{Hamiltonian system} is given by a pair $(\Pi_P,\delta)$, where $\delta$ is a section of the affine bundle $J^1P^* \to \Pi_P$.
Pulling back the canonical form $\Theta$ on $J^1P^*$ along $\delta$ yields the Hamiltonian forms
\[
\Theta_\delta := \delta^*\Theta,
\qquad
\Omega_\delta := -d\Theta_\delta,
\]
which encode the dynamics.
A section $p\colon M \to \Pi_P$ is said to be a \emph{solution} of the Hamiltonian system if
\begin{equation}\label{Hameq}
p^*(i_X \Omega_\delta)=0
\end{equation}
for every vertical vector field $X\in\mathfrak{X}(\Pi_P)$.
Locally, each $p\in \Pi_P$ can be expressed in coordinates $(x^i,y^a,p^i_a)$ as
\[
p = p^i_a \frac{\partial}{\partial x^i} \otimes dy^a \otimes d^n x,
\]
and the section $\delta$ takes the form
\[
\delta(x^i,y^a,p^i_a)
= (x^i,y^a,p^i_a,H_\delta(x^i,y^a,p^i_a)).
\]
In these coordinates, the pulled-back canonical form reads
\begin{equation}\label{pullback canform}
\Theta_\delta
= p^i_a \, dy^a \wedge d^{n-1}x_i
+ H_\delta(x^i,y^a,p^i_a)\, d^n x.
\end{equation}

An Ehresmann connection $\Lambda$ on $P \to M$ induces a section
$\delta_\Lambda$ of $J^1P^* \to \Pi_P$ given by
\[
\delta_\Lambda(v_x \otimes \omega_y \otimes \vol)
= (\omega_y \circ \Lambda) \wedge i_{v_x}\vol\in Z_y\cong (J^1P^*)_y, 
\]
for $v_x \otimes \omega_y \otimes \vol
\in T_xM \otimes V_y^*P \otimes \bigwedge\nolimits^n T_x^*M$.
The difference $\mathcal{H} := \delta - \delta_\Lambda$ defines a map
$\mathcal{H}\colon \Pi_P \to \bigwedge\nolimits^n T^*M$, called the \emph{Hamiltonian density}.
Accordingly, a Hamiltonian system may equivalently be described by a triple
$(\Pi_P,\Lambda,\mathcal{H})$.

The local expression of \eqref{Hameq} yields the \emph{Hamilton--Cartan equations}
\begin{equation}\label{Hamilton-Cartan}
\frac{\partial H}{\partial p^i_a}
= \frac{\partial y^a}{\partial x^i} - \Lambda^a_i,
\qquad
-\frac{\partial H}{\partial y^a}
= \frac{\partial p^i_a}{\partial x^i}
+ \frac{\partial \Lambda^b_i}{\partial y^a}\, p^i_b,
\end{equation}
where $\Lambda^a_i$ denote the local coefficients of $\Lambda$ and
$\mathcal{H} = H\, d^n x$.
These equations admit an equivalent formulation in terms of a covariant Poisson bracket, as we now recall.

A differential form $F$ on $J^1P^*$ is said to be \emph{horizontal} if it vanishes upon contraction
with vertical vectors with respect to the projection $J^1P^* \to M$, that is,
\[
F = F_{i_1\cdots i_r}\, dx^{i_1} \wedge \cdots \wedge dx^{i_r}.
\]

An horizontal $r$-form $F$ is called \emph{Poisson} if there exists a vertical
$(n-r)$-multivector field $\chi_F$ on $J^1P^*$ such that
\begin{equation}
i_{\chi_F}\Omega = dF.
\end{equation}
Given Poisson forms $F$ and $E$ of degrees $r$ and $s$, respectively, their bracket
is defined as the $(r+s+1-n)$-form
\begin{equation}\label{Poisson bracket}
\{F,E\} = (-1)^{r(s-1)}\, i_{\chi_E} i_{\chi_F}\Omega,
\end{equation}
which is again a Poisson form.
As shown in \cite{ForgerPoisson, PoissonForms}, this operation defines a graded Poisson bracket
with a modified Leibniz rule.
In particular, Poisson $(n-1)$-forms descend to $\Pi_P$ and provide a bracket formulation
of the Hamilton--Cartan equations \eqref{Hamilton-Cartan}, in close analogy with the role
of affine functions on cotangent bundles in classical Hamiltonian mechanics. More precisely:

\begin{theorem}\emph{\cite[Proposition 5.2]{someremarks}} \label{th: bracket formulation}
A section $\pi$ of $\Pi_P\to M$ is a solution of a given Hamiltonian system $(\Pi_P,\Lambda,\mathcal{H})$, $\mathcal{H}=H\vol$ if and only if for any horizontal Poisson $(n-1)$-form $F$ the following equation holds true:
\begin{equation} \label{Hameqs with bracket}
\{F,H\}\vol\circ \pi=d(\pi^*F)-(d^hF)\circ \pi,
\end{equation}
where $d^hF$ is the horizontal differential of $F$ with respect to the connection on $\Pi_P$.
\end{theorem}

\begin{remark} \label{rmk: conn. poly}
The connection on $\Pi_P\to M$ metioned in Theorem \ref{th: bracket formulation} is defined from the Ehresmann connection $\Lambda$ on $P\to M$ and any linear connection $\Gamma$ on $M$. Indeed, we know from \cite{kolarNatural,saundersGeometryJetBundles} that $\Lambda$ naturally extends to $VP\to M$ and $V^*P\to M$. Then, the linear connection on $M$ provides the extention to $\Pi_P=TM\otimes V^*P\otimes\bigwedge\nolimits^nT^*M$. In local coordinates, the corresponding horizontal lift is
\begin{align} \label{eq: conn.poly}
\frac{\partial}{\partial x^i}\mapsto\frac{\partial}{\partial x^i}+\Lambda^a_i\frac{\partial}{\partial y^a}+\left(-\frac{\partial \Lambda^b_i}{\partial y^a}p^j_b+\Gamma^j_{ik}p^k_a-\Gamma^k_{ik}p^j_a\right)\frac{\partial}{\partial p^j_a}.
\end{align}
\end{remark}
Poisson $(n-1)$-forms on $\Pi_P$ admit a particularly simple and useful description. 

\begin{proposition}\label{prop: n-1}\emph{\cite[Proposition 4.3]{someremarks}}
Any Poisson $(n-1)$-form on $\Pi_{P}$ can be written as
\begin{equation} \label{eq: n-1 forms}
F=\theta_X+\pr^*_{P,\Pi_P}\omega+\Upsilon,
\end{equation}
where $\omega$ is any horizontal $(n-1)$-form on $P$, $\Upsilon$ is an arbitrary closed form on $\Pi _P$, and $\theta _X$ is the horizontal $(n-1)$-form defined by an arbitrary vertical vector field $X\in \Gamma (VP)$ on $P$ as
\[
(\theta _X)_q= \alpha (X)\,\,i_u \nu
\]
for decomposable elements $q\in \Pi _P=TM\otimes V^*P\otimes \bigwedge\nolimits\nolimits ^n T^*M$, $q=u\otimes\alpha\otimes \nu$, and extended linearly to $\Pi_P$.
\end{proposition}

In order to study the reduced polysymplectic space we will repeatedly use the following result which follows from \cite[Ch. II, §3]{nomizu}.

\begin{lemma}\label{lemma: vector quotient}
Let $P\to M$ a $G$-principal bundle and let $V \to P$ be a vector bundle on which $G$ acts equivarantly, that is, the action commutes with the vector bundle projection. Then the map
\begin{align}
\psi \colon V &\longrightarrow V/G \times_M P, \nonumber \\ 
v_p &\longmapsto \bigl([v_p]_G,\, p\bigr),
\end{align}
is a vector bundle isomorphism over $P$ and a fiber diffeomorphism over $M$.
\end{lemma}

\section{Reduced polysymplectic bundle} \label{sec: reduced polysymplectic}

As previously mentioned, the main objective of this paper is to study Hamiltonian field theories on principal bundles whose symmetry group is a subgroup of the structure group. In this Section, we focus on the geometric properties of the reduced polysymplectic space $\Pi_{P} / K$ and explore its identification with $\Pi_{P}/G \times_{M} P/K$.

We will first introduce local coordinates which reflect the structure of the relevant bundles. Let $P\to M$ be a $G$-principal bundle, and $K\subset G$ a closed Lie subgroup. We denote $\mathfrak{g}$ and $\mathfrak{k}$ their respective Lie algebras. Let $\{\mathbf{B}_{\alpha}\}_{\alpha\in1,\dots,k}$ be a basis of $\mathfrak{k}$, and $\{\mathbf{B}_{\alpha}, \mathbf{B}_A\,\}_{\alpha\in1,\dots,k\,; A\in1,\dots,m-k}$ be an extension that forms basis of $\mathfrak{g}$, the structure constants are 
$$
\begin{aligned}
& \left[\mathbf{B}_{\alpha}, \mathbf{B}_{\beta}\right]=c_{\alpha \beta}^{\gamma} \mathbf{B}_{\gamma} \quad \\
& \left[\mathbf{B}_{\alpha}, \mathbf{B}_{A}\right]= c^{\gamma}_{\alpha A}  \mathbf{B}_{\gamma}+ c^C_{\alpha A} \mathbf{B}_{C} \\
& \left[\mathbf{B}_{A}, \mathbf{B}_{B}\right]=c_{A B}^{\gamma} \mathbf{B}_{\gamma}+c_{A B}^{C} \mathbf{B}_{C}
\end{aligned}
$$
where the first relation reflects that $[\mathfrak{k}, \mathfrak{k}] \subseteq \mathfrak{k}$.

Now consider $\left(x^{i}, y^{A}, k^{\alpha}\right)$ a normal coordinate system in a neighborhood of $p\in P$ in the sense that in the local trivialization, $p=(x,e)$ and given $(x,g)$ with coordinates $y^A(g),k^{\alpha}(g)$, then \begin{equation*}
    g=\exp(y^A(g)\mathbf{B}_A+k^{\alpha}(g)\mathbf{B}_{\alpha}).
\end{equation*}
Note that $\left(x^{i}, y^{A}\right)$ are also fiber coordinates adapted to $P/K\to M$. Furthermore, as described in Section \ref{sec: multi}, coordinates on $P$ induce local coordinates on the corresponding multimomentum bundle $\Pi_{P}$. In this context, the induced coordinates are $\left(x^{i}, y^{A}, k^{\alpha}, \pi_{A}^{i}, \pi_{\alpha}^{i}\right)$ and the Hamilton--Cartan equations are expressed as follows:
\begin{align*}
\frac{\partial H}{\partial \pi_{A}^{i}} &= \frac{\partial y^{A}}{\partial x^{i}} - \Lambda_{i}^{A} 
&\quad -\frac{\partial H}{\partial y^{A}} &= \frac{\partial p_{A}^{i}}{\partial x^{i}} + \frac{\partial \Lambda_{i}^{B}}{\partial y^{A}} \pi_{B}^{i} + \frac{\partial \Lambda_{i}^{\beta}}{\partial y^{A}} \pi_{\beta}^{i} \\
\frac{\partial H}{\partial \pi_{\alpha}^{i}} &= \frac{\partial k^{\alpha}}{\partial x^{i}} - \Lambda_{i}^{\alpha} 
&\quad -\frac{\partial H}{\partial k^{\alpha}} &= \frac{\partial p_{\alpha}^{i}}{\partial x^{i}} + \frac{\partial \Lambda_{i}^{B}}{\partial k^{\alpha}} \pi_{B}^{i} + \frac{\partial \Lambda_{i}^{\beta}}{\partial k^{\alpha}} \pi_{\beta}^{i},
\end{align*}
where $\Lambda_{i}^{A}\left(x^{j}, y^{\beta}, k^{\beta}\right)$, $\Lambda_{i}^{\alpha}\left(x^{j}, y^{\beta}, k^{\beta}\right)$ are the coefficients of a $G$-principal connection on $P \to M$.

We now study the reduced polysymplectic space $\Pi_{P} / K$. 

\begin{proposition} \label{prop: map}
    Let $P\to M$ be a $G$-principal bundle, and $K$ a closed Lie subgroup of $G$. Then the mapping     
\begin{align} \label{eq: map}
\varPsi \colon \Pi_{P} / K &\longrightarrow \Pi_{P}/G \times_{M} P/K \\
[\pi_x]_K &\longmapsto \left(\mu=[\pi_x]_G, \bar{s}=[\pr_{P,\Pi_P}(\pi_x)]_K\right) \nonumber
\end{align}
    is a fiber diffeomorphism.
\end{proposition}
\begin{proof}
    From Lemma \ref{lemma: vector quotient}, $\Pi_{P}$ and $\Pi_{P}/G \times_{M} P$ are diffeomorphic via $\psi(\pi_x)=(\mu,\pr_{P,\Pi_P}(\pi_x))$. As $K$ acts trivially on $\Pi_{P}/G$ considering the quotient by the action of $K$ concludes the proof.
\end{proof}
From \eqref{eq:intro-someremarks}, Proposition \ref{prop: map} provides the identification
$$\Pi_{P}/ K\cong \left(T M \otimes \g^{*} \otimes \bigwedge\nolimits^{n} T^*M\right)\times_{M} P/K.$$
To simplify notation, we will refer to both $\Pi_{P}/G$ and $T M \otimes \g^{*} \otimes \bigwedge\nolimits^{n} T^*M$ interchangeably, when no confusion arises. We denote by
$\Psi=\varPsi\circ \pr_{\Pi_P/K,\Pi_P}$, the projection from $\Pi_P$ to $\Pi_P/K\simeq\Pi_{P}/G \times_{M} P/K$. This decomposition is key to perform Lie--Poisson reduction and admits a geometric interpretation. Indeed, since sections of $P/K\to M$ are in one-to-one correspondence with reductions of the structure group of $P\to M$ from $G$ to $K$, the above splitting implies that studying $K$-invariant sections of $\Pi_P$ is equivalent to studying $G$-invariant sections together with reductions of $P\to M$ from $G$ to $K$.

In the coordinate system introduced above, the projection $\Psi$ takes the expression:
\begin{align} \label{eq: map local}
\Psi \colon \Pi_{P} \longrightarrow& \Pi_{P}/G \times_{M} P/K \\
\left(x^{i}, y^{A}, k^{\alpha}, \pi_{A}^{i}, \pi_{\alpha}^{i}\right) \longmapsto&\left(x^{i}, \mu_{A}^{i}=Z_{A}^{B}\left(g^{-1}\right) \pi_{B}^{i}+Z_{A}^{\beta}\left(g^{-1}\right) \pi_{\beta}^{i},
\right.\nonumber
\\ &
\left.\qquad\mu_{\alpha}^{i}=Z_{\alpha}^{B}\left(y^{-1}\right) \pi_{B}^{l}+Z_{\alpha}^{\beta}\left(g^{-1}\right) \pi_{\beta}^{i}, y^{A}\right),
 \nonumber
\end{align}
where functions $Z^I_J:U\subseteq P\to\mathfrak{g}$ are determined by
\begin{align}
& \operatorname{T}_g \operatorname{R}_{g^{-1}}\left(\frac{\partial}{\partial y^A}\right)_g=Z_A^B(g)\left(\frac{\partial}{\partial k^B}\right)_e+Z_A^\beta(g)\left(\frac{\partial}{\partial y^\beta}\right)_e, \\
& \operatorname{T}_g \operatorname{R}_{g^{-1}}\left(\frac{\partial}{\partial k^\alpha}\right)_g=Z_\alpha^B(g)\left(\frac{\partial}{\partial k^B}\right)_e+Z_\alpha^\beta(g)\left(\frac{\partial}{\partial y^\beta}\right)_e.
\end{align}
These functions were previously studied in the proof of \cite[Lemma 6.4]{someremarks}. There, a lengthy but standard argument involving Baker–Campbell–Hausdorff formula shows that
\begin{equation}\label{eq: derivative of Z}
    \frac{\partial Z^I_J}{\partial y^{K}}=-\frac{1}{2}c^{I}_{JK}.
\end{equation}

\begin{proposition} \label{prop: diff identification}
The local expression of $T\Psi:T\Pi_P \to T(\Pi_{P}/G \times_{M} P/K)$ at point $p=(s,e)$ is
\begin{align*}
T\Psi: T\Pi_P &\to T(\Pi_{P}/G \times_{M} P/K)\\
\frac{\partial}{\partial x^i}&\mapsto \frac{\partial}{\partial x^i} \\
\frac{\partial}{\partial y^A}&\mapsto \frac{1}{2}\mu^{j}_{\gamma}c^{\gamma}_{\beta A}\frac{\partial}{\partial \mu^j_{\beta}}+\frac{1}{2}\mu^{j}_{C}c^{C}_{\beta A}\frac{\partial}{\partial \mu^j_{\beta}}+ \frac{1}{2}\mu^{j}_{\gamma}c^{\gamma}_{BA}\frac{\partial}{\partial \mu^j_{B}}+\frac{1}{2}\mu^{j}_{C}c^{C}_{BA}\frac{\partial}{\partial \mu^j_{B}}+\frac{\partial}{\partial y^A} \\
\frac{\partial}{\partial k^{\alpha}}&\mapsto \frac{1}{2}\mu^{j}_{\gamma}c^{\gamma}_{\beta\alpha}\frac{\partial}{\partial \mu^j_{\beta}}+\frac{1}{2}\mu^{j}_{\gamma}c^{\gamma}_{B\alpha}\frac{\partial}{\partial \mu^j_{B}}+\frac{1}{2}\mu^{j}_{C}c^{C}_{B\alpha}\frac{\partial}{\partial \mu^j_{B}}\\
\frac{\partial}{\partial \pi^i_a}&\mapsto \frac{\partial}{\partial \mu^i_A}\\
\frac{\partial}{\partial \pi^i_{\alpha}}&\mapsto \frac{\partial}{\partial \mu^i_{\alpha}}\\
\end{align*}
\end{proposition}
\begin{proof}
    From the local expression \eqref{eq: map local} of $\Psi$, the result follows using the chain rule and \eqref{eq: derivative of Z} provided by \cite[Lemma 6.4]{someremarks}.
\end{proof}

\begin{lemma} \label{lemma: invariance}
Let $E:\Pi_P\to \R$ be a $K$-invariant real function on $\Pi_P$, then
\begin{equation}
    0=\frac{\partial E}{\partial k^{\alpha}}
    -\frac{1}{2}\pi^i_{\gamma}c^{\gamma}_{\beta\alpha}\frac{\partial E}{\partial \pi^i_\beta}
    -\frac{1}{2}\pi^i_{\gamma}c^{\gamma}_{B\alpha}\frac{\partial E}{\partial \pi^i_B}
    -\frac{1}{2}\pi^i_{C}c^{C}_{B\alpha}\frac{\partial E}{\partial \pi^i_B}.
\end{equation}
Similarly, let $D: P\to \mathfrak{g}$ be a $K$-equivariant function on $P$, then
\begin{align}
    0&=\frac{\partial D^{\gamma}}{\partial k^{\alpha}}-\frac{1}{2}c^{\gamma}_{\beta\alpha}D^{\beta}-\frac{1}{2}c^{\gamma}_{B\alpha}D^{B},\\
    0&=\frac{\partial D^{C}}{\partial k^{\alpha}}-\frac{1}{2}c^{C}_{\beta\alpha}D^{\beta}-\frac{1}{2}c^{C}_{B\alpha}D^{B},
\end{align}
\end{lemma}
\begin{proof}
Since $E$ is invariant under the action of $K$ on $\Pi_P$, for any $k\in K$,
\begin{align*}
    E\left(x^i, y^{A}, 0, \pi_{B}^{i}, \pi_{\beta}^{i}\right)=&E\left(k \cdot\left(\dot{x}^{i}, y^{A}, 0, \pi_{B}^{i}, \pi_{\beta}^{i}\right)\right)\\=&E(x^i,y^{A}, k^{\lambda}, Z^{\gamma}_{B}(k)\pi^i_{\gamma}+Z^{C}_{B}(k)\pi^i_{C},Z^{\gamma}_{\beta}(k)\pi^i_{\gamma}+Z^{C}_{\beta}(k)\pi^i_{C}).
\end{align*}
Now consider the one-parameter subgroup $k(\varepsilon)=\exp(\varepsilon B_{ \alpha})$ for which $k^{\alpha}=\varepsilon$, $k^{\beta}=0$ when $\beta\neq\alpha$. Then, 
$$E(x^i,y^A,0,\pi_{B}^{i}, \pi_{\beta}^{i})=E(x^i,y^A, k^{\lambda}(\varepsilon), p^i_a,Z^{\gamma}_{\beta}(g(\varepsilon))p^i_{\gamma}),$$
Differentiating both sides with respect to $\varepsilon$ at $\varepsilon=0$, and applying the identity \eqref{eq: derivative of Z}, we obtain the first part of the lemma. For the second part, the equivariance of $D$ implies that, for any $k\in K$, 
$$D^{\gamma}(x^i,y^A, k^{\lambda})=Z^{\gamma}_{\beta}(k^{-1})D^{\beta}(x^i,y^A,0)+Z^{\gamma}_{B}(k^{-1})D^{B}(x^i,y^A,0).$$
In particular, it holds for $k(\varepsilon)=\exp(\varepsilon B_{ \alpha})$ and derivation with respect to  $\varepsilon$ completes the proof.
\end{proof}

\section{Lie--Poisson Reduction} \label{sec: reduction}
This section contains the core results of the paper as it describes how to perform Lie--Poisson reduction in the multisymplectic and polysymplectic formalism when the group of symmetry is a subgroup of the structure group of a principal bundle.

\begin{proposition} \label{prop: reduced forms}
Let $f$ be a $(n-1)$-form on $\Pi_{P}/K$ such that $F=\Psi^*f$ is a $K$-invariant Poisson $(n-1)$-form on $\Pi_P$. Then
\begin{equation}\label{eq: n-1 red}
f=\theta_{\bar{\xi}}+\pr^*_{P/K,\Pi_P/K}\omega+\Upsilon,
\end{equation}
where $\bar{\xi}\in\Gamma (\pr_{M,P/K}^*\g\to P/K)$, $\omega$ is an horizontal $(n-1)$-form on $P/K$ and $\Upsilon$ is a closed horizontal $(n-1)$-form on $\Pi_P/K$.
\end{proposition}
\begin{proof}
From Proposition \ref{prop: n-1}, any $K$-invariant Poisson $(n-1)$-form on $\Pi_P$ can be written as
\begin{equation*}
F=\theta_X+\pr^*_{P,\Pi_P}\bar{\omega}+\bar{\Upsilon},
\end{equation*}
where $X$ is a $K$-invariant section of $VP\to P$, $\bar{\omega}$ is an horizontal $K$-invariant $(n-1)$-form on $P$ and $\bar{\Upsilon}$ is a closed horizontal $K$-invariant $(n-1)$-form on $\Pi_P$. Since both $\bar{\omega}$ and $\bar{\Upsilon}$ are horizontal and  $K$-invariant, they descend to forms on the quotient spaces: $\bar{\omega}$ projects to a horizontal $(n{-}1)$-form $\omega$ on $P/K$, and $\bar{\Upsilon}$ projects to a horizontal $(n{-}1)$-form $\Upsilon$ on $\Pi_P/K$. Similarly, the $K$-invariant section $X$ induces a section $\bar{\xi}$ of the quotient bundle $VP/K \to P/K $. As $VP=P\times\mathfrak{g}$, it suffices to show that $(P\times\mathfrak{g})/K$ and $\pr_{M,P/K}^*\g$ are isomorphic bundles. Indeed,
    \begin{align}
        \alpha: (P\times\mathfrak{g})/K &\to \pr_{M,P/K}^*\g \nonumber\\
        [p,\xi]_K &\mapsto (\pr_{P/K,P}(p),[p,\xi]_G)
    \end{align}
is clearly a smooth bundle map over the identity whose inverse can be defined as follows: Given  $(y,[p,\xi]_G)\in \pr_{M,P/K}^*\g$, we have that $y\in (P/K)_x$ and $p\in P_x$ for some $x\in M$, but not necessarily $\pr_{P/K, P}(\tilde{p})=y$. Let $\tilde{p} \in P_{x}$ such that $\pr_{P/K, P}(\tilde{p})=y$, there exists $g \in G$ such that $\tilde{p}=R_g(p)$. Then
$\left(y,[p, \xi]_{G}\right)=\left(y,\left[\tilde{p}, \Ad_{g}, \xi\right]_{G}\right)$, and
$$
\alpha^{-1}\left(y,[p, \xi]_{G}\right)=\left[\tilde{p},\Ad_{g} \xi\right]_{K}.
$$
The well posedness of this definition is a straightforward calculation.
\end{proof}

The local expression of a reduced Poisson $(n-1)$-form on $\Pi_P/K$ is $f=f^i \iota_{\partial /\partial x^i}\mathrm{v},$ where 
\begin{equation}
    f^{i}=\mu_{A}^{i} \xi^{A}\left(x^{i}, y^{A}\right)+\mu_{\alpha}^{i} \xi^{\alpha}\left(x^{i}, y^{A}\right)+\omega^{i}\left(x^{i}, y^{A}\right).
\end{equation}

To define the reduced covariant bracket, we begin by introducing some necessary notation. Let $s : U \to P$ be a compactly supported (local) section, and $\bar{s} : U \to P/K$ be the induced reduced section. We define the first-order differential operator
\begin{align} 
    \mathsf{P}_{\bar{s}} : \Gamma(\g) &\longrightarrow  \Gamma(\bar{s}^* V(P/K)) \nonumber \\
\eta &\to \mathsf{P}_{\bar{s}}(\eta) = \eta_{P/K} \big|_{\mathrm{Im}\,\bar{s}}, \label{eq: Ps}
\end{align}
where  $\eta_{P/K}$ denotes the infinitesimal generator of the $\mathfrak{g}$-action on $P/K$. The adjoint or dual operator $\mathsf{P}^{+}_{\bar{s}} : \Gamma(\bar{s}^* V^*(P/K)) \longrightarrow \Gamma(\g^*)$
is defined by
\begin{equation}\label{eq: Ps adjoint}
   \mathsf{P}^{+}_{\bar{s}}(\Upsilon)(\eta) = \langle \Upsilon, \eta_{P/K} \rangle 
\end{equation}
for all 
$\Upsilon \in \Gamma(\bar{s}^* V^*(P/K))$ and $\eta \in \Gamma(\tilde{\mathfrak{g}})$, where \(\langle \cdot, \cdot \rangle\) denotes the duality pairing. These maps were previously introduced in \cite{EPsubgroup} to describe the propagation of the set of admissible infinitesimal variations in Euler--Poincar\'e reduction. 

For a smooth function \(h \in C^\infty(\Pi_P/K)\), we define two vertical derivatives corresponding to the splitting of 
$\Pi_P/K
\;\cong\;
\bigl(TM\otimes\tilde{\mathfrak g}^*\otimes\Lambda^{n}T^*M\bigr)
\times_M (P/K).
$
The \emph{vertical derivative with respect to the momentum variable} is the fiber derivative
\[
\frac{\delta h}{\delta \mu}
\;:\;
\Pi_P/K
\longrightarrow
T^*M \otimes \tilde{\mathfrak g}\otimes \bigwedge\nolimits ^nTM,
\]
characterized by the property that, for any
\(\delta\mu \in TM\otimes\tilde{\mathfrak g}^*\otimes\wedge ^{n}T^*M\),
\[
\Big\langle \frac{\delta h}{\delta \mu}, \delta\mu\Big\rangle :=\left.\frac{d}{d\varepsilon}\right|_{\varepsilon=0}
h(\mu+\varepsilon\,\delta\mu,\bar s),
\]
As well as, the \emph{vertical derivative with respect to the configuration variable} which is the derivative along the fibers of \(P/K\to M\),
\[
\frac{\delta h}{\delta \bar{s}}
\;:\;
\Pi_P/K
\longrightarrow
V^*_{\bar{s}}(P/K),
\]
defined by
\[
\Big\langle \frac{\delta h}{\delta \bar{s}},\, X \Big\rangle:=\left.\frac{d}{d\varepsilon}\right|_{\varepsilon=0}
h(\mu,\exp(\varepsilon X)\cdot \bar s)
\qquad
X \in V_{\bar s}(P/K).
\]

\begin{definition} \label{def: Reduced Bracket}
    Let $f$ be a reduced Poisson $(n-1)$-form on $\Pi_P/K$ and $h$ a Hamiltonian density on $\Pi_{P}/ K$. We define their bracket as
\begin{equation} \label{eq: Reduced Bracket}
\{f,h\}=\{f,h\}_{\mathrm{LP}}+ \{f,h\}_{E},
\end{equation}
where $\{f,h\}_{LP}$ is the Lie-Poisson bracket on $T M \otimes \g^{*} \otimes \bigwedge\nolimits^{n} T^*M$
\begin{equation} \label{eq: Lie-Poisson bracket}
\{\bar{\xi},h\}_{\mathrm{LP}}=-\left\langle\mu,\left[\bar{\xi},\frac{\delta h}{\delta\mu}\right]\right\rangle
\end{equation}
and
\begin{equation} \label{eq: bracket on E}
\{f,h\}_{E}=\left\langle\frac{\delta f}{\delta\bar{s}},\mathsf{P}_{\bar{s}}\left(\frac{\delta h}{\delta\mu}\right)\right\rangle-\left\langle\frac{\delta h}{\delta\bar{s}},\mathsf{P}_{\bar{s}}(\bar{\xi})\right\rangle.
\end{equation}
\end{definition}

\begin{proposition} \label{prop: reduced bracket}
Let $\Psi:\Pi_{P}\rightarrow\Pi_{P}/K\cong(T M \otimes \g^{*} \otimes \bigwedge\nolimits^{n} T^*M)\times_MP/K$ be the projection in Proposition \ref{prop: diff identification}. Then for any $K$-invariant Poisson $(n-1)$-form $F$ and any $K$-invariant Poisson function $H$ on $\Pi_P/K$
\begin{equation}
\{F,H\}=\Psi^*\{f,h\},
\end{equation}
where the bracket on the left hand side is defined by \eqref{Poisson bracket}, $f$ is a reduced Poisson $(n-1)$-form on $\Pi_{P}/K$ such that $\Psi^*f=F$, $h$ is a function such that $\Psi^*h=H$, and the bracket on the right hand side is defined by \eqref{eq: Reduced Bracket}.
\end{proposition}

\begin{proof}
    The local expression of $\left\{F,H\right\}$ in the coordinate system introduced in Section \ref{sec: reduced polysymplectic} is
    \begin{equation}
    \left\{F,H\right\}=\frac{\partial F^{i}}{\partial y^{A}} \cdot \frac{\partial H}{\partial \pi_{A}^{i}}-\frac{\partial F^{i}}{\partial \pi_{A}^{i}} \cdot \frac{\partial H}{\partial y^{A}}+\frac{\partial F^{i}}{\partial k^{\alpha}} \cdot \frac{\partial H}{\partial \pi_{\alpha}^{i}}-\frac{\partial F^{i}}{\partial \pi_{\alpha}^{i}} \cdot \frac{\partial H}{\partial k^{\alpha}}.
    \end{equation}
    Since $F$ and $H$ are $K$-invariant, from Proposition \ref{prop: diff identification} and Lemma \ref{lemma: invariance} the following function $r$ on $\Pi_{P}/K$ satisfies that $\{F,H\}=\Psi^*r$:
        \begin{align*}
          r=&\frac{1}{2} c_{\beta A}^\gamma \mu^j_\gamma \frac{\partial f^i}{\partial \mu_\beta^j} \frac{\partial h}{\partial \mu_A^i}
          +\frac{1}{2} c_{\beta A}^C \mu_C^j \frac{\partial f^i}{\partial \mu^j_\beta} \frac{\partial h}{\partial \mu_A^i}
          +\frac{1}{2} c_{B A}^\gamma \mu_{\gamma}^j \frac{\partial f^i}{\partial \mu_B^j} \cdot \frac{\partial h}{\partial \mu_A^i}+\frac{1}{2} c_{B A}^C \mu_C^j \frac{\partial f^i}{\partial \mu_B^j} \frac{\partial h}{\partial \mu_A^i}
          +\frac{\partial f^i}{\partial y^A} \frac{\partial h}{\partial \mu_A^i} \\
          &-\frac{\partial f^i}{\partial \mu_A^i}\left(
          +\frac{1}{2} c_{\beta A}^\gamma \mu_\gamma^j \frac{\partial h}{\partial \mu^j_\beta}
          +\frac{1}{2} c_{B A}^\gamma \mu^j_\gamma \frac{\partial h}{\partial \mu_B^j}
          +\frac{1}{2} c_{\beta A}^C \mu_C^j \frac{\partial h}{\partial \mu_\beta^j}
          +\frac{1}{2} c_{B A}^C \mu_C^j \frac{\partial h}{\partial \mu_B^j}\right)
          -\frac{\partial f^i}{\partial \mu_A^i} \frac{\partial h}{\partial y^A} \\
          & +\frac{1}{2} c_{\beta \alpha}^{\gamma} \mu^{j}_\gamma \frac{\partial f^{i}}{\partial \mu_{\beta}^j} \frac{\partial h}{\partial \mu_{\alpha}^{i}}
          +\frac{1}{2} c_{B \alpha}^{\gamma} \mu^{j}_\gamma \frac{\partial f^{i}}{\partial \mu_{B}^{j}} \frac{\partial h}{\partial \mu_{\alpha}^{i}}
          +\frac{1}{2} c_{B \alpha}^{C} \mu_{C}^{j} \frac{\partial f^{i}}{\partial \mu^j_{B}} \frac{\partial h}{\partial \mu_{\alpha}^{i}} \\
          & -\frac{\partial f^i}{\partial \mu_\alpha^i}\left(+\frac{1}{2} c_{\beta \alpha}^\gamma \mu^j_\gamma \frac{\partial h}{\partial \mu_\beta^j}
          +\frac{1}{2} c_{B \alpha}^\gamma \mu_\gamma^j \frac{\partial h}{\partial \mu_B^j}
          +\frac{1}{2} c_{B \alpha}^C \mu_C^j \frac{\partial h}{\partial \mu_B^j}\right)
        \end{align*}
    We shall see that $r=\{f,h\}$ as desired. As $f$ is a reduced Poisson form, from Proposition \ref{prop: reduced forms}
    \begin{equation*}
        \frac{\partial f^{i}}{\partial \mu^j_{\alpha}}=\xi^{\alpha} \delta_{j}^{i},
        \qquad
        \frac{\partial f^{i}}{\partial \mu^j_{A}}=\xi^{A} \delta_{j}^{i}.
    \end{equation*}
    Therefore,
    \begin{align} 
        r=& +c_{\beta A}^\gamma \mu_\gamma^{j} \frac{\partial f^i}{\partial \mu_\beta^j} \frac{\partial h}{\partial \mu^i_A}
        +c_{B A}^C \mu_C^{j} \frac{\partial f^i}{\partial \mu_\beta^{j}} \frac{\partial h}{\partial \mu^i_\alpha}
        +c_{B A}^\gamma \mu_\gamma^{j} \frac{\partial f^i}{\partial \mu_B^{j}} \frac{\partial h}{\partial \mu_A^{i}}
        +c_{B A}^C \mu^j_C\frac{\partial f^i}{\partial \mu^j_B} \frac{\partial h}{\partial \mu_A^{i}} \nonumber \\
        & +c_{\beta \alpha}^\gamma \mu_\gamma^{j} \frac{\partial f^i}{\partial \mu^j_\beta} \frac{\partial h}{\partial \mu_\alpha^i}
        +c_{B \alpha}^\gamma \mu_\gamma^j \frac{\partial f^i}{\partial \mu_B^{j}} \frac{\partial h}{\partial \mu^i_\alpha}
        +c_{B \alpha}^C \mu^j_C\frac{\partial f^i}{\partial \mu_B^{j}} \frac{\partial h}{\partial \mu_\alpha^{i}}
        +\frac{\partial f^i}{\partial y^A} \frac{\partial h}{\partial \mu_A^{i}}
        -\frac{\partial f^i}{\partial \mu_A^{i}} \frac{\partial h}{\partial y^A} \label{eq: local bracket}
    \end{align}
    which is the local expression of $\{f,h\}$.
\end{proof}


\begin{remark} \label{rmk: conn. red}
    Given a $K$- invariant Hamiltonian system $(\Pi_P,\Lambda,\mathcal{H})$ a connection on $\Pi_P/K\to M$ is naturally defined. Indeed, the principal connection $\Lambda$ on $P\to M$ induces a connection in the associated bundle $\g^*$. Combined with a linear connection $\Gamma$ on $M$, this yields a connection on $\Pi_P/G\simeq TM\otimes \g^*\otimes\bigwedge\nolimits^nT^*M$ (See Remark \ref{rmk: conn. poly}). This is precisely the connection used to describe dynamics in in \cite{someremarks}. Furthermore, $\Lambda$ induces a connection $\bar{\Lambda}$ in the associated bundle $P/K\to M$. Together, these two induced connections determine a connection on $\Pi_P/K\simeq\Pi_P/G\times_MP/K$.
\end{remark}

\begin{proposition} \label{prop: reduced dh}
    Let $F$ be a $K$-invariant Poisson $(n-1)$-form on $\Pi_P$, and let $f$ be a reduced Poisson $(n-1)$-form on $\Pi_P/K$ such that $\Psi^*f=F$. Then,
        \begin{equation}
            d^hF= \Psi^*\left(d^hf\right),
        \end{equation}
    where $d^hF$ and $d^hf$ are the horizontal differentials with respect to the connection on $\Pi_P$ introduced in \eqref{eq: conn.poly}, and the induced connection on $\Pi_P/K$ described in Remark \ref{rmk: conn. red}.
\end{proposition}
\begin{proof}
Geometrically, the horizontal subspaces of $\bar{\Lambda}$ in Remark \ref{rmk: conn. red} are just the projection from $P$ to $P/K$ of the horizontal distribution defined by $\Lambda$. Similarly, the horizontal subspaces of the natural extension of $\Lambda$ to $VP$ projects to the horizontal subspaces of the associated adjoint connection on $VP/G\simeq\g$.
Thus, 
\begin{equation}
\Hor_{(\mu,\bar{s})}=T\Psi\circ \Hor_\pi
\end{equation}
where $\Hor_\bullet$ denotes the horizontal lift to $\pi$ and $(\mu,\bar{s})$ using the respective connections on $\Pi_P$ and $\Pi_P/K$. Given $\pi$ a section of $\Pi_P\to M$ and $(\mu,\bar{s})=\Psi\circ \pi$ section of $\Pi_p/K\to M$, we find that
$$
\begin{aligned}
d^{h} f \circ (\mu,\bar{s})&=df\circ \Hor_{(\mu,\bar{s})}  =d f\circ T\Psi \circ \Hor_{\pi}=d\left(f \circ \Psi\right) \circ \Hor_{\pi} \\
&=d F \circ \Hor_{\pi}=d^{h} F \circ \pi .
\end{aligned}
$$
\end{proof}


\begin{remark} \label{rmk: connection sigma}
    The bundle $\mathcal{C}=J^1P/G$ of $G$-principal connections on $P\to M$ is an affine bundle over $M$ modeled by the vector bundle $T^*M\otimes\g\to M$. Furthermore, given a section $(\mu,\bar{s})$ of $\Pi_P/K\to M$, the vertical derivative 
    $$\frac{\delta h}{\delta \mu}: \Pi_P/K \to T^*M\otimes\g$$
    can be pullbacked to a section of $T^*M\otimes\g\to M$. Thus, 
    $$\sigma=(\mu,\bar{s})^*\left(\frac{\delta h}{\delta \mu}\right)+\Lambda$$ is a $G$-principal connection on $P\to M$.
\end{remark}

\begin{theorem} \label{th: reduced dynamics}
Let $P\to M$ be a $G$-principal bundle over a manifold $M$ with a volume form $\vol$ and let $K$ be a closed subgroup of $G$. Let $\Lambda$  be a  $G$-principal connection on $P\to M$ and $\mathcal{H}=H\vol$ a $K$-invariant Hamiltonian density on $\Pi_P$. Denote by $\mathcal{h}=h\vol$ the reduced Hamiltonian density and for any section $\pi$ of $\Pi_P\to M$, let $(\mu,\bar{s})=\Psi \circ \pi$ be the reduced section of
$$\Pi_P/K\cong\left(TM\otimes\g^*\otimes\bigwedge\nolimits^{n-1}T^*M\right)\times_MP/K\to M.$$
Then, the following are equivalent:
\begin{enumerate}[\em (i)\em]
\item for every Poisson $(n-1)$-form $F$ on $\Pi_P$, the following identity holds true:
\begin{equation*}
\{F,H\}\vol=d(F\circ \pi)-d^{h}F\circ \pi.
\end{equation*}
\item the section $\pi:M\to\Pi_P$ satisfies the Hamilton--de Donder equations,
\item for every reduced Poisson $(n-1)$-form $f$ on $\Pi_P/K$,
\begin{equation}\label{intrinsic.red.eq}
\{f,h\}\vol=d(f\circ(\mu,\bar{s}))-d^{h}f\circ(\mu,\bar{s}),
\end{equation}
where the bracket is defined by Equation \eqref{eq: Reduced Bracket}.
\item the section $(\mu,\bar{s}):M\to\Pi_P/K$ satisfies the equations
\begin{equation}\label{eq: Lie-Poisson}
  \mathrm{div}^{\Lambda}\mu-\mathrm{ad}^*_{\delta h /\delta\mu}\mu+P^{+}_{\bar{s}}\left(\frac{\delta h}{\delta\bar{s}}\right)=0,
\end{equation}
\begin{equation}\label{eq: parallel}
  \nabla^{\sigma_K}\bar{s}=0,
\end{equation}
where $\sigma_K$ is the Ehresmann connection on $P/K\to M$ induced by $\sigma$, the connection introduced in Remark \ref{rmk: connection sigma}.
\end{enumerate}
\end{theorem}
\begin{proof}
    The equivalence (i)$\Leftrightarrow$(ii) is established in Theorem \ref{th: bracket formulation}. We now proceed to prove the equivalence (i)$\Leftrightarrow$(iii). From Propositions \ref{prop: reduced bracket} and \ref{prop: reduced dh},  it suffices to show that $d(F\circ \pi)=\Psi^*(d(f\circ(\mu\oplus\bar{s})))$. Indeed, since $F$ is horizontal, for any $v_1,\dots v_{n-1}$ vectors in $T_xM$
        \begin{align*}
        (\pi^*F)_x(v_1,\dots,v_{n-1})&=F_{\pi(x)}(T_x\pi(v_1),\dots,T_x\pi(v_{n-1}))=F_{\pi(x)}(v_1,\dots,v_{n-1})\\&=(F\circ \pi)_x(v_1,\dots,v_{n-1}).
        \end{align*}
    Similarly, $(\mu,\bar{s})^*f=f\circ(\mu,\bar{s})$. Thus,
    $$F\circ \pi=\pi^*F=\pi^*\Psi^*f=(\mu,\bar{s})^*f=f\circ(\mu,\bar{s}),$$
    and, as required, $d(F\circ p)=d(f\circ(\mu,\bar{s}))$.
    Finally, equivalence (iii)$\Leftrightarrow$(iv) is obtained as follows:
    The local expressions 
    \begin{equation} \label{eq: local df}
        d\left( f\circ(\mu, \bar{s})\right)=\frac{\partial f^{i}}{\partial x^{i}}+\frac{\partial f^{i}}{\partial y^{A}} \frac{\partial y^{A}}{\partial x^{i}}+\frac{\partial f^{i}}{\partial \mu^j_{A}} \frac{\partial \mu^j_{A}}{\partial x^{i}}+\frac{\partial f^{i}}{\partial \mu^j_{\alpha}} \frac{\partial \mu^j_{\alpha}}{\partial x^{i}},
    \end{equation}
    \begin{align*} \label{eq: local dhf}
        d^{h}f\circ(\mu,\bar{s})=&\frac{\partial f^i}{\partial x^i}
    +\frac{\partial f^i}{\partial y^A} \Lambda_i^A +\frac{\partial f^i}{\partial \mu_A^j}\left(
    -\left(\frac{\partial \Lambda_i^B}{\partial y^A}-\frac{1}{2} C_{A C}^B \Lambda_i^C-\frac{1}{2}C^B_{A\gamma} \Lambda_i^\gamma\right) \mu_B^j\right) \nonumber \\&
    +\frac{\partial f^i}{\partial \mu_A^j}\left(
    -\left(\frac{\partial \Lambda_i^\beta}{\partial y^A}-\frac{1}{2} C_{A C}^\beta \Lambda_i^C-\frac{1}{2} C_{A \gamma}^\beta \Lambda_i^\gamma\right) \mu^j_\beta
    +\Gamma_{i k}^j \mu_A^k-\Gamma_{i k}^k \mu_\alpha^j\right) \nonumber \\
    & +\frac{\partial f^{i}}{\partial \mu_{\alpha}^{j}}\left(
    -C_{C \alpha}^{B} \Lambda_{i}^{C} \mu_{\beta}^{j}
    -C_{\gamma\alpha}^{\beta} \Lambda_{i}^{\gamma} \mu_{\beta}^{j}
    -C_{C \alpha}^{\beta} \Lambda_{i}^{C} \mu_{\beta}^{j}
    +\Gamma_{i k}^{j} \mu_{\alpha}^{k}
    -\Gamma_{i k}^{k} \mu_{\alpha}^{j}\right),
    \end{align*}
together with equation \eqref{eq: local bracket} provide a local description of \eqref{intrinsic.red.eq}. Note that, as    
    \begin{equation*}
        \frac{\partial f^{i}}{\partial \mu^j_{\alpha}}=\xi^{\alpha} \delta_{j}^{i},
        \qquad
        \frac{\partial f^{i}}{\partial \mu^j_{A}}=\xi^{A} \delta_{j}^{i}.
    \end{equation*}
terms involving Christoffel symbols $\Gamma$ become irrelevant. As equation \eqref{intrinsic.red.eq} is true for any reduced Poisson $(n-1)$-form, we can group terms on $\frac{\partial f^{i}}{\partial y^{A}}$ and obtain equation;
    \begin{equation} \label{eq: parallel local}
        \frac{\partial h}{\partial \mu_{A}^i}=\frac{\partial y^{A}}{\partial x^{i}}-\Lambda_{i}^{A}
    \end{equation}
which is the local expression of Equation \eqref{eq: parallel}. Similarly, grouping terms on  $\frac{\partial f^{i}}{\partial \mu_{A}^{i}}$ and $\frac{\partial f^{i}}{\partial \mu_{\alpha}^{i}}$,

\begin{multline}\label{eq: Lie-Poisson local A}
\mu_{\gamma}^{i} C_{A B}^{\gamma} \frac{\partial h}{\partial \mu_{B}^{i}}
+\mu_{C}^{i} C_{A B}^{C} \frac{\partial h}{\partial \mu_{B}^i}
+\mu_{\gamma}^{i} C_{A \beta}^{\gamma} \frac{\partial h}{\partial \mu_{\beta}^{i}}
+\mu_{C}^{i} C_{A \beta}^{C} \frac{\partial h}{\partial \mu_{\beta}^i}
-\frac{\partial h}{\partial y^{A}}
\\
= \frac{\partial \mu_{A}^{i}}{\partial x^{i}}
+\left(\frac{\partial \Lambda_{i}^{B}}{\partial y^{A}}
-\frac{1}{2} C_{A C}^{B} \Lambda_{i}^{C}
-\frac{1}{2} C_{A \gamma}^{B} \Lambda_{i}^{\gamma}\right) \mu^i_{B}
+\left(\frac{\partial \Lambda_{i}^{\beta}}{\partial y^{A}}
-\frac{1}{2} C_{A C}^{\beta} \Lambda_{i}^{C}
-\frac{1}{2} C_{A \gamma}^{\beta} \Lambda_{i}^{\gamma}\right) \mu^i_\beta,
\end{multline}
    \begin{equation} \label{eq: Lie-Poisson local alfa}
    \mu_{\gamma}^{i} C_{\alpha B}^{\gamma} \frac{\partial h}{\partial \mu_{B}^i}
    + \mu_{\gamma}^{i} c_{\alpha \beta}^{\gamma} \frac{\partial h}{\partial \mu_{\beta}^i}
    + \mu_{\gamma}^{i} c_{\alpha B}^{\gamma} \frac{\partial h}{\partial \mu_{B}^i} 
    = \frac{\partial \mu_{\alpha}^{i}}{\partial x^{i}} \\
    + C_{C \alpha}^{B} \Lambda_{i}^{C} \mu_{B}^{i}
    + C_{\gamma \alpha}^{\beta} \Lambda_{i}^{\gamma} \mu^i_{\beta}
    + C^{\beta}_{C \alpha} \Lambda_{i}^{C} \mu_{\beta}^i,
    \end{equation}
which are the local expression of Equation \eqref{eq: Lie-Poisson}.
\end{proof}

\section{Reconstruction} \label{sec: reconstruction}

We now study the link between solutions of the reduced Hamiltonian system and those of the original problem. Theorem \ref{th: reduced dynamics} shows that a solution $\pi(x)$ of a Hamiltonian system with group of symmetry $K$ can be projected to a pair $(\mu(x),\bar{s}(x))$ satisfying the reduced equations. A natural question then arises: under what conditions can one reconstruct a solution of the original system from a given reduced solution $(\mu(x), \bar{s}(x))$? We will see that obstruction to reconstruction is determined by the flatness of the connection introduced in Remark \ref{rmk: connection sigma}.


\begin{theorem} \label{th: reconstruction}
Let $P \to M$ be a principal $G$-fiber bundle over an oriented simply-connected manifold $M$ with a volume form $\vol$, and let $K$ be a closed subgroup of $G$. Let $\mathcal{H}=H\vol$, where $H : \Pi_P \to \R$, be a $K$-invariant Hamiltonian, and let $\Lambda$ be a $G$-principal connection on $P\to M$. A solution $(\mu, \bar{s})$ of the Lie--Poisson equations \eqref{eq: Lie-Poisson} and \eqref{eq: parallel} is the reduction of a solution $\pi:M\to \Pi_P$ of the original Hamiltonian system defined by $(\Pi_P,\Lambda,\mathcal{H})$, if and only if 
$$\sigma=(\mu,\bar{s})^*\left(\frac{\partial h}{\partial \mu}\right)+\Lambda$$ is a flat connection with trivial holonomy.

\medskip
For non-simply-connected manifolds, as the holonomy of any flat connection is locally trivial, we always have the local equivalence
\[
\left.
\begin{array}{r}
 \pi \text{ satisfies the} \\
 \text{Hamilton-de Donder} \\
 \text{equations of } (\Pi_P, \Lambda, \mathcal{H}) 
\end{array}
\right\}
\Longleftrightarrow 
\left\{
\begin{aligned}
    &\mathrm{div}^{\Lambda}\mu 
    - \mathrm{ad}^*_{\delta h / \delta \mu} \mu 
    + P^{+}_{\bar{s}}\left( \frac{\delta h}{\delta \bar{s}} \right) = 0, \\
    &\nabla^\sigma \bar{s} = 0, \\
    &\Curv(\sigma) = 0.
\end{aligned}
\right.
\]
\end{theorem}

\begin{proof}
Suppose that $\pi$ is a section of $\Pi_P\to M$ which is a solution of $(\Pi_P,\Lambda,\mathcal{H})$. Consider $s(x)=\pr_{P,\Pi_P}\pi(x)$ a section of $P\to M$ and $(\mu,\bar{s})=\Psi \circ \pi$, the reduced section of $\Pi_P/K$. Then, 

$$
\begin{aligned}
 \nabla^{\sigma} s=&\left(\frac{\partial y^{A}}{\partial x^{i}}-\sigma_{i}^{A}\left(x^{j}\right)\right) \frac{\partial}{\partial y^{A}} +\left(\frac{\partial k^{\alpha}}{\partial x^{i}}-\sigma_{i}^{\alpha}\left(x^{j}\right)\right) \frac{\partial}{\partial k^{\alpha}} \\
=&\left(\frac{\partial y^{A}}{\partial x^{i}}-\frac{\partial h}{\partial \mu_{A}^{i}}\left(\mu\left(x^{j}\right),\bar{s}\left(x^{j}\right)\right)-\Lambda_{i}^{A}\left(x^{j}\right)\right) \frac{\partial}{\partial y^{A}} \\
&+ \left(\frac{\partial k^{\alpha}}{\partial x^{i}}-\frac{\partial h}{\partial \mu_{\alpha}^{j}}\left(\mu\left(x^{j}\right), \bar{s}\left(x^{j}\right)\right)-\Lambda_{i}^{\alpha}\left(x^{j}\right)\right) \frac{\partial}{\partial k^{\alpha}} \\
=&\left(\frac{\partial y^{A}}{\partial x^{i}}-\frac{\partial H}{\partial p^{i}_A}\left(\pi\left(x^{j}\right)\right)-\Lambda_{i}^{A}\left(x^j\right)\right) \frac{\partial}{\partial y^{A}}
+
\left(\frac{\partial k^{\alpha}}{\partial x^{i}}-\frac{\partial H}{\partial p^{i}_\alpha}\left(\pi\left(x^{j}\right)\right)-\Lambda_{i}^{\alpha}\left(x^{j}\right)\right) \frac{\partial}{\partial k^{\alpha}} =0 \\
\end{aligned}
$$
where, in the second equivalence, we have used that from Proposition \ref{prop: diff identification},
$$(\mu,\bar{s})^*\left(\frac{\partial h}{\partial \mu^i_A}\right)=\pi^*\left(\Psi^*\left(\frac{\partial h}{\partial \mu^i_A}\right)\right)=\pi^*\left(\frac{\partial H}{\partial \pi^i_A}\right).$$
Furthermore, in the third equivalence we have used that $\pi$ satisfies the Hamilton--de Donder equations. Since $s(x)$ is parallel with respect to $\sigma$, we conclude that $\Curv(\sigma)=0$.

Given $(\mu, \bar{s})$ solutions of the reduced Hamiltonian system and the flat connection $\sigma$, we shall construct a section $\pi(x)$ of $\Pi_P\to M$ that solves the original problem such that $(\mu, \bar{s})=\Psi\circ\pi$. First, we obtain a section $s\in\Gamma(P\to M)$ that projects to $\bar{s}\in\Gamma(P/K\to M)$. Since $\Curv(\sigma)=0$ and the holonomy is trivial, the integral leaves of $\sigma$ are given by $\hat{s} \cdot g$, $g \in G$, for certain section $\hat{s}$ of $P\to M$.  
Fix a point $x_0 \in M$. Then, there exists $g \in G$ such that the image of $\hat{s}(x_0)g$ under the projection $\pr_{P/K, P}$ coincides with $\bar{s}(x_0)$, that is,
\[
\pr_{P/K, P}(\hat{s}(x_0)g) = \bar{s}(x_0).
\]
Define $s := \hat{s} \cdot g$. Since $s$ is an integral leaf of $\sigma$, it follows that $\pr_{P/K, P}(s)$ is a parallel section of $P/K \to M$ with respect to the induced connection $\sigma_K$. By the Frobenius theorem, the sections $\bar{s}$ and $\pr_{P/K, P}(s)$ must coincide, as they are both parallel and agree at a single point. 

Then, we construct $\pi(x)$ as the only element in $\left(\Pi_P\right)_{s(x)}$ that projects to $\mu(x).$ This is always possible as a direct application of the reduction of vector bundles over a principal bundle to the following diagram.

\begin{center}
\begin{tikzpicture}[scale=1.5] \label{quo.vect.ble}
\node (0) at (0,1.2) {$\Pi_P$};
\node (A) at (3,1.2) {$\Pi_P/G\cong\left(TM\otimes\g^*\otimes\bigwedge\nolimits^{n-1}T^*M\right)$};
\node (B) at (3,0) {$M=P/G$};
\node (C) at (0,0) {$P$};
\draw[->,font=\scriptsize,>=angle 90]
(0) edge node[above]{} (A)
(0) edge node[right]{} (C)
(C) edge node[above]{} (B)
(A) edge node[right]{} (B);
\draw[->, blue, >=angle 90, bend right=20] 
(B) to node[right] {$\mu(x)$} (A);

\draw[->, blue, dashed, >=angle 90] 
(B) to node[xshift=2pt, yshift=3pt, right] {$\pi(x)$} (0);

\draw[->, blue, >=angle 90, bend left=10] 
(B) to node[below right] {$s(x)$} (C);
\end{tikzpicture}
\end{center}

From Theorem \ref{th: reduced dynamics}, as $(\mu,\bar{s})=\Psi\circ\pi$ and $(\mu,\bar{s})$ solves the reduced Hamiltonian system, $\pi(x)$ is a solution of the original problem as desired.

\end{proof}

\begin{remark}
    Let $\pi$ be solution of the unreduced system recosntructed from $(\mu, \bar{s})$ with the procedure of Theorem \ref{th: reconstruction}. As the Hamiltonian system is $K$-invariant, the sections of the type $R_k \circ \pi$, $k \in K$, are solutions of the unreduced problem. In fact, these are all the solutions of the unreduced problem projecting to $(\mu, \bar{s})$.
    Furthermore, from the proof of Theorem \ref{th: reconstruction}, 
    $$s(x)=\pr_{P,\Pi_P}(\pi(x))$$
    is an integral leaf of $\sigma$ and $R_k \circ s$, $k \in K$, are all the integral leaves of $\sigma$ projecting to $\bar{s}$.
\end{remark}

Theorem \ref{th: reconstruction} resembles the reconstruction results shown in \cite{EPsubgroup} for the Lagrangian picture. In that paper, given
$L : J^1P \to \mathbb{R}$ a $K$-invariant Lagrangian and 
$$l:J^1P/K\cong\mathcal{C}\times_M(P/K)\to\R,$$ 
the reduced Lagrangian, it is shown that a solution $(\sigma, \bar{s})$ of the Euler-–Poincar\'e equations for $l$ is the reduction of a solution of $s$ of the original variational problem defined by $L$ if and only if $\sigma$ is a flat connection with trivial holonomy and $\bar{s}$ is parallel with respect to $\sigma$. In analogy with the final result from Theorem \ref{th: reconstruction}, for non-simply-connected manifolds, as the holonomy of any flat connection is locally trivial, they obtain the local equivalence
\[
\mathcal{E}\mathcal{L}(L)(s) = 0 \Longleftrightarrow 
\begin{cases}
\mathcal{E}\mathcal{P}(l)(\sigma, \bar{s}) = 0 \\
\mathrm{Curv}(\sigma) = 0 \\
\nabla^\sigma \bar{s} = 0.
\end{cases}
\]

The main difference between the Hamiltonian picture and the Lagrangian counterpart in \cite{EPsubgroup} lies in how the reconstruction conditions are incorporated. In that work, the reduction of the variational problem yields only the Euler–Poincaré equations, while the conditions $\mathrm{Curv}(\sigma) = 0$ and $\nabla^\sigma \bar{s} = 0$ appear separately as part of the reconstruction process. In contrast, the present paper shows that when reducing the Hamilton--de Donder equations, the condition $\nabla^\sigma \bar{s} = 0$ emerges from the reduction itself, and only $\mathrm{Curv}(\sigma) = 0$ must be imposed as an additional reconstruction condition.


\begin{remark}
    The compatibility condition in Theorem \ref{th: reconstruction} is local in the case of non-simply-connected manifolds. This is illustrated in the following example. Consider the additive group $G = \R^2$ and $K = \R$, its first component. Let $M = S^1$ and $P = M \times G \to M$, whose sections are functions $s(\theta) = (\theta, x(\theta), y(\theta))$. Consider the Hamiltonian 
    $$H(\theta; x, y, \pi_x, \pi_y) = \frac{1}{2} \left((\pi_x)^2 + (\pi_y)^2 - y^2\right),$$
    which is invariant under translations along the $x$-axis. Since $P$ is a product, we take the connection $\Lambda$ to be trivial. The resulting Hamilton--de Donder equations are:
    \begin{alignat}{2}
    \pi_x &= x' \quad &\quad 0 &= \pi_x' \label{eq: nsc Hamilton x} \\
    \pi_y &= y' \quad &\quad y &= \pi_y'. \label{eq: nsc Hamilton y}
    \end{alignat}
    and are straightforward to solve. Equations \eqref{eq: nsc Hamilton x} imply $x'' = 0$. Given the $2\pi$-periodicity of $x(\theta)$, we conclude that $x(\theta) = x_0$ is constant, and hence $\pi_x(\theta) = 0$. Similarly, Equations \eqref{eq: nsc Hamilton y} imply $y'' - y = 0$, and the periodicity condition imposes $y(\theta) = 0$, so $\pi_y(\theta) = 0$.

The reduced Hamiltonian is defined on
$$(\Pi_P/\R) \times_{S^1} (S^1 \times \mathbb{R}) = \left(TS^1 \otimes (\R^2)^{*} \otimes T^*S^1\right)\times_{S^1} (S^1 \times \mathbb{R})$$
as
$$h(\theta; \mu_x, \mu_y, y) = \frac{1}{2} \left(\mu_x^2 + \mu_y^2 - y^2\right),$$
where $\bar{s} = y$. As the group $G$ is abelian, the reduced Hamilton equations are
\[
\frac{d}{d\theta} \frac{\delta h}{\delta \mu} + P^+_{\bar{s}} \left(\frac{\delta h}{\delta \bar{s}}\right) = 
\left( \frac{d \mu_x}{d\theta}, \frac{d \mu_y}{d\theta} \right) + (0, -y) = 0,
\]
which are simply,
\begin{equation} \label{eq: nsc Hamiltonian red}
\frac{d \mu_x}{d\theta} = 0, \qquad \frac{d \mu_y}{d\theta} = y;
\end{equation}
together with the parallel condition $\nabla^{\sigma_K} \bar{s} = 0$. Since $\sigma = (\mu_x d\theta, \mu_y d\theta)$ and $\sigma_K = \mu_y d\theta$, the latter becomes
\begin{equation} \label{eq: nsc parallel}
y' - \mu_y = 0.
\end{equation}
The $2\pi$-periodic solutions of \eqref{eq: nsc Hamiltonian red} and \eqref{eq: nsc parallel} are $\mu_x(\theta) = \mu_0$ constant, $\mu_y = 0$ and $y = 0$.
As $\dim M = 1$, the compatibility condition $\mathrm{Curv}(\sigma) = 0$ is trivially satisfied and provides no additional constraints. Hence, if  $\mu_x(\theta) = \mu_0 \neq 0$, there is no solution of the original problem projecting to it.
\end{remark}

\section{Examples} \label{sec: examples}

\subsection{Heavy Top}\label{subsec: heavy top}

In this Subsection, we apply the Lie--Poisson reduction developed above to a well-known system in classical Mechanics: the heavy top. The resulting equations of motion are compatible with the dynamics described in \cite{MRbook}. Consider the trivial principal bundle $P = \mathbb{R} \times SO(3) \to \mathbb{R} = M$ and the subgroup $K = SO(2)$ of the structure group $G= SO(3)$. The polysymplectic space of this system is
\begin{equation*}
    \Pi_P = (T\R \otimes V^*P\otimes T^*\R ) = \mathbb{R} \times T^*(SO(3)),
\end{equation*}
and sections of this bundle are denoted as $(t,R(t),\pi(t))$. In accordance with identification \eqref{eq: map local}, since $S^2 = SO(3)/SO(2)$,  the reduced polysymplectic space is 
\begin{equation*}
    (\Pi_P)/K = (T\R \otimes \mathfrak{so}^*(3)\otimes T^*\R) \times_{\R} (\R \times S^2) = \mathbb{R} \times (\mathfrak{so}^*(3) \times S^2),
\end{equation*}
and sections $(\sigma, \bar{s})$ of this bundle are written as
\[
(\sigma, \bar{s})(t) = (t, \mu(t), \Gamma(t))
\]
for certain curves
\[
\mu : \mathbb{R} \to \mathfrak{so}^*(3), \quad \Gamma : \mathbb{R} \to S^2.
\]
In fact, $\mathfrak{so}(3)$, and by duality $\mathfrak{so}^*(3)$, are identified with $\R^3$ in the standard way
\begin{equation*}
    \begin{pmatrix}
    0 & c & b \\
    -c & 0 & a \\
    -b & -a & 0
    \end{pmatrix}
\mapsto (a,b,c),
\end{equation*}
and under this identification the Lie Bracket transforms into the cross product. 

We now introduce the Hamiltonian \( H: TSO(3) \to \mathbb{R} \) that models the (right invariant) heavy top, which is defined as
\begin{equation} \label{eq: heavy top H}
    H(R, \pi) = \frac{1}{2} \llangle\pi, \mathbb{I}^{-1}\pi\rrangle + mg \langle R \cdot \mathbf{e}_3, \chi \rangle
\end{equation}
where $\llangle \cdot, \cdot \rrangle$ denotes a right invariant metric in $SO(3)$, \( \mathbb{I} \) is the (right) inertia tensor,
\( \mathbf{e}_3 \in \mathbb{R}^3  \) corresponds to the third vector of the canonical basis interpreted as the vertical direction  with respect to the gravity field, \( \chi \) is the vector joining the fixed point of the top with its center of mass, and $\langle \cdot, \cdot \rangle$ represents the Euclidean inner product in $\mathbb{R}^3$. Although the classical formulation of the heavy top is typically expressed using a left-invariant Hamiltonian, we adopt the right-invariant perspective here to remain consistent with the formulation developed in the preceding sections. Nevertheless, analogous results can be readily derived in the left-invariant setting.

The reduced Hamiltonian $h: \mathbb{R} \times (\mathfrak{so}^*(3) \times S^2) \to \mathbb{R}$ is
\begin{equation} \label{eq: heavy top h}
    h( \mu, \Gamma) = \frac{1}{2} \llangle \mu, \mathbb{I}^{-1}\mu \rrangle + mg \langle \Gamma, \chi \rangle,
\end{equation}
and a curve $ (t, R(t), \pi(t)) \in \Pi_P \simeq \mathbb{R} \times T^*SO(3)$ projects under $\Psi$ in Proposition \ref{prop: map} to the pair of curves
\[
(t, R(t), \pi(t)) \mapsto (\mu = \pi \cdot R^{-1}, \quad \Gamma = R \cdot e_3).
\]
The equations of motion of the reduced system are easily obtained from Theorem \ref{th: reduced dynamics}. On one hand, the adjoint operator $\mathsf{P}^{+}_{\bar{s}}$ in Equation \eqref{eq: Ps adjoint} particularizes to
\begin{align*} \label{eq: Heavy Top adjoint}
    \mathsf{P}_{\Gamma}^+: C^\infty(\mathbb{R}, T^*S^2) &\to C^\infty(\mathbb{R}, \mathfrak{so}(3)^*) \\
    \Upsilon &\mapsto \Gamma \times \Upsilon.
\end{align*}
Then, as $\frac{\delta h}{\delta \mu}=\mathbb{I}^{-1}\mu,$ Equation \eqref{eq: Lie-Poisson} reads as
\begin{equation} \label{eq: Heavy Top Lie-Poisson}
    \frac{d\mu}{dt} + (\mathbb{I}^{-1}\mu) \times \mu + mg \Lambda \times \chi = 0.
\end{equation}
On the other hand, as $\Lambda$ is trivial, $\sigma$ is defined by $\mathbb{I}^{-1}\mu\in\Gamma(T^*M\otimes\R^3\to M)$ and the parallel Equation \eqref{eq: parallel} particularizes as 
\begin{equation} \label{eq: Heavy Top parallel}
    \frac{d\Gamma}{dt} +  \left(\mathbb{I}^{-1}\mu\right) \times \Gamma = 0.
\end{equation}
Finally, the compatibility condition stated in Theorem \ref{th: reconstruction}, namely  \( \text{Curv}(\sigma) = 0 \), is identically satisfied in this case as $ \dim M = 1 $.

\subsection{\texorpdfstring{$SO(3)$-strand with Broken Symmetry}{SO(3)-strand with Broken Symmetry}}
\label{subsec:strand}
The preceding example can therefore be treated entirely within the framework of classical mechanics. We now turn to an example that genuinely lies within the domain of field theory, where $ \dim M > 1 $.

A $G$-strand is a map $R(s,t)\colon \R\times\R\to G$ arising from a class of $G$-invariant Hamiltonians that generalize a vast amount of classical chiral models. See \cite{matrixgstrands,gstrands}.
The reduction procedure developed in the previous sections provides a framework for analyzing $G$-strands when an extra term breaks the symmetry, that is, the group of symmetry is a subgroup $K$ of $G$. In this Section, we study an $SO(3)$-strand, a physically relevant example that models spins chains or a molecular strand as in \cite{MolStrand}. A symmetry breaking term, which is reminiscent of the heavy top, is added to model an external uniform electric field. 

Consider a trivial principal bundle \( P = \mathbb{R}^2 \times SO(3) \to \mathbb{R}^2 \) with structure group \( G = SO(3) \). We can take local coordinates \( (s,t) \) in \( \mathbb{R}^2 \) (\( s \) for space and \( t \) for time), identify sections of $P$ with maps \( R: \mathbb{R}^2 \to SO(3) \), and the sections of the polysymplectic bundle
\begin{equation*}
    \Pi_P
= T\R^2 \otimes V^*P \otimes \bigwedge\nolimits^2 T\R^2
= T\R^2 \otimes T^{*}SO(3) \otimes \bigwedge\nolimits^2 T\R^2,
\end{equation*}
with functions $\pi^s,\pi^t\colon \R^2\to T^*SO(3)$. Let $\Lambda$ be the trivial Ehresmann connection on the product bundle $P$, $\vol=ds\wedge dt$, and consider the following Hamiltonian
\begin{equation} \label{eq: Hamiltonian strand}
    H\bigl(s,t,R,\pi^s,\pi^t\bigr)
    = -\frac{1}{2}\,\llangle \pi^s, \mathbb{J}^{-1}\pi^s \rrangle
    + \frac{1}{2}\,\llangle \pi^t, \mathbb{I}^{-1}\pi^t \rrangle
    + m g \,\langle R \mathbf{e}_3, \chi \rangle,
\end{equation}
where $\mathbb{J}$ is a tensor describing the opposition to the rotation of consecutive rigid bodies of the strand. The rest of the notation is shared with the Heavy Top example in Subsection \ref{subsec: heavy top}. The first two terms of Hamiltonian \eqref{eq: Hamiltonian strand} are $SO(3)$ invariant while the last one is simply $SO(2)$ invariant then the group of symmetries is $K=SO(2)$. Since $P/K=SO(3)/SO(2)=S^2 $, from identification \eqref{eq: map local}, the reduced polysymplectic space is
\begin{equation*} \label{eq: red Hamiltonian strand}
    \Pi_P / K
    = \bigl( T\R^2 \otimes \mathfrak{so}^*(3) \otimes \bigwedge\nolimits^2 T^*\R^2 \bigr)
    \times_{\R^2} \bigl( \R^2 \times S^2 \bigr)\simeq \mathbb{R}^2\times (\mathbb{R}^2\otimes \mathfrak {so}^*(3)\times S^2),
\end{equation*}
and momenta $\pi^s,\pi^t$ project to $\mu^s=\pi^sR^{-1}$, $\mu^t=\pi^tR^{-1}$, and $\Gamma=R \mathbf{e}_3$. Hence, the reduced Hamiltonian is the following,
\begin{equation}
   h\bigl(s,t,\mu^s,\mu^t,\Gamma\bigr)
    = -\frac{1}{2}\,\llangle \mu^s, \mathbb{J}^{-1}\mu^s \rrangle
    + \frac{1}{2}\,\llangle \mu^t, \mathbb{I}^{-1}\mu^t \rrangle
    + m g\,\langle \Gamma, \chi \rangle. 
\end{equation}

We shall now study the reduced equations of motion. On one hand, Equation \eqref{eq: Lie-Poisson} reads as
\begin{equation} \label{eq: Strand Lie-Poisson}
     \frac{d\mu^s}{ds} + \frac{d\mu^t}{dt} - (\mathbb{J}^{-1}\mu^s) \times \mu^s + (\mathbb{I}^{-1}\mu^t) \times \mu^t + mg \Lambda \times \chi = 0.
\end{equation}
On the other hand, as $\Lambda$ is trivial, the connection defined in Remark \ref{rmk: connection sigma} is
\begin{equation*}
    \sigma
= \bigl( \mu \oplus \Gamma \bigr)^*\!\left(\frac{\delta h}{\delta \mu}\right) + \Lambda
= \left( \frac{\partial}{\partial s} - \mathbb{J}^{-1}\mu^s \right)\! ds
  + \left( \frac{\partial}{\partial t} + \mathbb{I}^{-1}\mu^t \right)\! dt
\end{equation*}
and the parallel Equation \eqref{eq: parallel} provides the conditions
\begin{equation} \label{eq: Strand parallel 1}
    \frac{d\Gamma}{ds} -  \left(\mathbb{J}^{-1}\mu^s\right) \times \Gamma = 0.
\end{equation}
\begin{equation} \label{eq: Strand parallel 2}
    \frac{d\Gamma}{dt} +  \left(\mathbb{I}^{-1}\mu^t\right) \times \Gamma = 0.
\end{equation}
Finally, the reconstruction condition stated in Theorem \ref{th: reconstruction}, namely  \( \text{Curv}(\sigma) = 0 \), is 
\begin{equation} \label{eq: Strand reconstruction}
    \mathbb{J}^{-1}\frac{d\mu^t}{ds} + \mathbb{I}^{-1}\frac{d\mu^t}{ds}-\mathbb{J}^{-1}\mu^s \times \mathbb{I}^{-1}\mu^t=0, 
\end{equation}
which is consistent with the reconstruction condition obtained in the Lagrangian setting for $SO(3)$-strands in \cite{matrixgstrands} expressed using the body angular velocity  $\Omega=\mathbb{J}^{-1}\mu^s$ and the body angular strain $\omega=\mathbb{I}^{-1}\mu^t$.
\subsection{Reduction in Affine principal bundles} \label{subsec: affine}

Let $\pi : P \to M$ be a $G$--principal bundle and let $G \times V \to V$ be a left linear representation of $G$ on a vector space $V$. Denote by
\[
E := (P \times V)/G \to M
\]
the associated vector bundle and consider the affine group 
$G_{\mathrm{aff}} := G \ltimes V$
with group law
\begin{equation} \label{eq: aff action}
    (g,v)\cdot(g',v') = (gg',\, gv' + v),
    \qquad g,g' \in G,\; v,v' \in V.
\end{equation}
Its Lie algebra is
$\mathfrak{g}_{\mathrm{aff}} = \mathfrak{g} \oplus V$
with bracket
\begin{equation} \label{eq: aff bracket}
    [(B,v),(B',v')] = ([B,B'],\, Bv' - B'v),
\qquad B,B' \in \mathfrak{g},\; v,v' \in V.
\end{equation}
The action of $G_{\mathrm{aff}}$ on $V$ is affine in the sense that
\[
(g,v)\cdot u = gu + v.
\]
Observe that \(G\) can be regarded as a closed subgroup of the affine group
\(G_{\mathrm{aff}}\) through the canonical embedding \(g \mapsto (g,0)\).

The \em affine principal bundle defined by $P \to M$ and the representation of $G$ on $V$ \em is a $G_{\mathrm{aff}}$--principal bundle with total space
\[
P_{\mathrm{aff}} := P \times_M E \to M
\]
and right action
\begin{equation} \label{eq: aff cov action}
    R_{(g,v)}(u_x,e_x) = (u_x g,\, e_x + [u_x,v]_G),
\end{equation}
where $[u_x,v]_G$ denotes the class of $(u_x,v)$ in $E_x$. $G$--invariant Field Theories   
on $P_{\mathrm{aff}}$ have already been studied, both from the Lagrangian \cite{affinered,EPsubgroup} and Hamiltonian perspective \cite{HamAffine}. Our aim is to frame the reduction of field theories on affine principal bundles within the reduction procedure presented in this paper.

Before advancing to studying dynamics, we shall first study some geometric aspects of $P_{\mathrm{aff}}$. The adjoint bundle of $P_{\mathrm{aff}}$ is naturally identified with
$\tilde{\mathfrak{g}}_{\mathrm{aff}}
\simeq \tilde{\mathfrak{g}} \oplus E$,
and the induced bracket on sections $(\eta,\xi),(\eta',\xi')$ is
\begin{equation} \label{eq: aff cov bracket}
[(\eta,\xi),(\eta',\xi')] =
([\eta,\eta'],\, \eta\cdot\xi' - \eta'\cdot\xi).
\end{equation}
Moreover, the quotient of $P_{\mathrm{aff}}$ by the linear subgroup $G \subset
G_{\mathrm{aff}}$ is given by
\[
P_{\mathrm{aff}}/G
= (P_{\mathrm{aff}} \times (G_{\mathrm{aff}}/G))/G_{\mathrm{aff}}
\simeq (P_{\mathrm{aff}} \times V)/G_{\mathrm{aff}},
\]
We denote this associated affine bundle by
$$
E_{\mathrm{aff}} := (P_{\mathrm{aff}}\times V)/G_{\mathrm{aff}}
$$
as it is canonically identified with $E$, regarded as an affine bundle modeled on itself.

Let $\omega_{\mathrm{aff}}$ be a $G_{\mathrm{aff}}$--principal connection $1$--form on
$P_{\mathrm{aff}} \to M$.
Using the decomposition
$\mathfrak{g}_{\mathrm{aff}} = \mathfrak{g} \oplus V$,
the pullback of $\omega_{\mathrm{aff}}$ along the bundle inclusion $i:  P \hookrightarrow P_{\mathrm{aff}}$ splits as
\begin{equation} \label{eq: affine conn split}
    i^*\omega_{\mathrm{aff}} = \omega + h,
\end{equation}
where $\omega$ is a principal connection $1$--form on $P \to M$, and $h$ is a $V$--valued tensorial $1$--form on $P$. Such form $h$ is known to descend to a $1$--form on $M$ with values in $E$, see for instance \cite[Ch. II]{nomizu}. Thus, there is a bijection between affine connections on $P_{\mathrm{aff}}$
and pairs $(\sigma,h)$ consisting of a connection $\sigma$ on $P$
and a $1$--form $h \in \Omega^1(M,E)$.

Let $\sigma_{\mathrm{aff}}=(\sigma,h)$ be an affine connection.
For a section $\bar{s}_{\mathrm{aff}} \in \Gamma(E_{\mathrm{aff}})$,
identified with $\bar{s}\in\Gamma(E)$, one has
\begin{equation} \label{eq: aff connection}
    \nabla^{\sigma_{\mathrm{aff}}}\bar{s}_{\mathrm{aff}}
    = \nabla^\sigma \bar{s} + h.
\end{equation}
For a section $(\eta,\xi)\in\Gamma(\tilde{\mathfrak{g}}_{\mathrm{aff}})$,
the covariant derivative reads
\begin{equation} \label{eq: aff derivative}
\nabla^{\sigma_{\mathrm{aff}}}(\eta,\xi)
= \bigl(\nabla^\sigma\eta,\; \nabla^\sigma\xi - \eta\cdot h\bigr),
\end{equation}
and provided 
$\mu_{\mathrm{aff}} = (\mu,\zeta)
\in \Gamma(TM\otimes\tilde{\mathfrak{g}}_{\mathrm{aff}}^*)
\simeq \Gamma(TM\otimes\tilde{\mathfrak{g}}^*)
\oplus \Gamma(TM\otimes E^*),$
the divergence operator satisfies
\begin{equation} \label{eq: aff divergence}
\operatorname{div}^{\sigma_{\mathrm{aff}}}\mu_{\mathrm{aff}}
=
\bigl(
\operatorname{div}^\sigma \mu + h\otimes\zeta,\;
\operatorname{div}^\sigma \zeta
\bigr),
\end{equation}
where $h\otimes\zeta\in \Gamma(E\otimes E^*)$ coupling the $T^*M$ and $TM$ parts of $h$ and $\zeta$. Furthermore, the bundle $\Gamma(E \otimes E^*)$ is naturally embedded into
$\Gamma(\tilde{\mathfrak g}^{\,*})$ by the pairing
\begin{equation} \label{eq: EtimesE*}
\langle e \otimes e^*,\eta \rangle
:= \langle e^*,\, \eta \cdot e \rangle,
\qquad
\forall\, e \in \Gamma(E),\;
e^* \in \Gamma(E^*),\;
\eta \in \Gamma(\tilde{\mathfrak g}),   
\end{equation}
where $\eta \cdot e$ denotes the infinitesimal action of
$\tilde{\mathfrak g}$ on $E$.

For a section $\bar{s}_{\mathrm{aff}} \in \Gamma(E_{\mathrm{aff}})$, which we
identify with a section $\bar{s} \in \Gamma(E)$, the operator \eqref{eq: Ps} takes the form
\begin{align} \label{eq: aff Ps}
    \mathsf{P}_{\bar{s}_{\mathrm{aff}}} \colon
\Gamma(\tilde{\mathfrak{g}}_{\mathrm{aff}})
= \Gamma(\tilde{\mathfrak{g}}) \oplus \Gamma(E)
&\to \Gamma(\bar{s}^* V E_{\mathrm{aff}})
= \Gamma(E)  \nonumber \\
(\eta,\xi)
&\mapsto \eta \cdot \bar{s} + \xi,
\end{align}
where $\eta \cdot \bar{s}$ denotes the infinitesimal action of
$\tilde{\mathfrak{g}}$ on $E$, and the corresponding adjoint operator  \eqref{eq: Ps adjoint} is given by
\begin{align} \label{eq: aff Ps adjoint}
    \mathsf{P}^+_{\bar{s}_{\mathrm{aff}}} \colon \Gamma(E^*) 
    &\mapsto \Gamma(\tilde{\mathfrak{g}}_{\mathrm{aff}}^*)
    = \Gamma(\tilde{\mathfrak{g}}^*) \oplus \Gamma(E^*) \nonumber \\
    \omega &\mapsto (\omega \otimes \bar{s},\, \omega),
\end{align}
where $\omega \otimes \bar{s}$ is interpreted as a section of
$\tilde{\mathfrak{g}}^*$ as described in \eqref{eq: EtimesE*}.

Let $\mathcal{H}=H\vol$ be a $G$-invariant Hamiltonian on $P_{\mathrm{aff}}$, and let $\Lambda$ be a connection on $P_{\mathrm{aff}}\to M$ obtained from the pullback of a $G$-principal connection on $P\to M$, also denoted as $\Lambda$. That is, $\Lambda$ has no $h$ component in the decomposition \eqref{eq: aff connection}. From identification \eqref{eq: map local}, the reduced polysymplectic space is 
\begin{align}
\Pi_{P_{\mathrm{aff}}} / G &\simeq\left(T M \otimes \g_{\mathrm{aff}} \otimes \wedge^{n-1}T^*M\right) \times_{M}\left(P_{\mathrm{aff}} / G\right) \nonumber\\
&\simeq\left(T M \otimes\left(\g^* \oplus E^{*}\right) \otimes \wedge^{n-1}T^*M\right) \times_{M} E_{\mathrm{aff}},
\end{align}
and we denote by $(\mu, \omega, \bar{s})$ any element of $\Pi_{P_{\mathrm{aff}}} / G$, where $\mu\in TM \otimes\g^*\otimes \wedge^{n-1}T^*M$, $\omega\in TM \otimes E^*\otimes \wedge^{n-1}T^*M$, and $s\in E_{\mathrm{aff}}$.

Denote by $h$ the reduced Hamiltonian density in $\Pi_{P_{\mathrm{aff}}} / G$. As seen in Definition \ref{def: Reduced Bracket}, the reduced bracket of $h$ with a reduced Poisson $(n-1)$-form  $f$ determined by $(\eta, \xi) \in \Gamma(\g \oplus E)$ is the sum of two covariant brackets. From the splitting $\tilde{\mathfrak{g}}_{\mathrm{aff}}\simeq \tilde{\mathfrak{g}} \oplus E$, we can also write
$$\frac{\delta h}{\delta \mu_{\mathrm{aff}}}=\left(\frac{\delta h}{\delta \mu}, \frac{\delta h}{\delta \omega}\right).$$
Thus,
\begin{align}
\{f,h\}_{\mathrm{LP}} &= \left\langle(\mu, \omega),\left[(\eta, \xi),\left(\frac{\delta h}{\delta \mu}, \frac{\delta h}{\delta \omega}\right)\right]\right\rangle =\left\langle(\mu, \omega),\left(\left[\eta, \frac{\delta h}{\delta \mu}\right], \eta \cdot \frac{\delta h}{\partial \omega}-\frac{\delta h}{\delta \mu}\cdot \xi\right)\right\rangle \nonumber\\
& =\left\langle\mu,\left[\eta, \frac{\delta h}{\delta \mu}\right]\right\rangle-\left\langle\xi \otimes \omega, \frac{\delta h}{\delta \mu}\right\rangle+\left\langle\frac{\delta h}{\delta \omega} \otimes \omega, \eta\right\rangle
\end{align}
Furthermore,
\begin{align}
\{f,h\}_{E}&=\left\langle\frac{\delta f}{\delta\bar{s}},\mathsf{P}_{\bar{s}}\left(\frac{\delta h}{\delta \mu}, \frac{\delta h}{\delta \omega}\right)\right\rangle-\left\langle\frac{\delta h}{\delta\bar{s}},\mathsf{P}_{\bar{s}}(\eta,\xi)\right\rangle \nonumber\\
&=\left\langle\frac{\delta f}{\delta\bar{s}},
\frac{\delta h}{\delta \mu}\cdot\bar{s}+ \frac{\delta h}{\delta \omega}
\right\rangle-\left\langle\frac{\delta h}{\delta\bar{s}},\eta\cdot\bar{s}+\xi\right\rangle \nonumber \\
&=\left\langle\frac{\delta f}{\delta\bar{s}}\otimes\bar{s},\frac{\delta h}{\delta \mu}\right\rangle 
+\left\langle\frac{\delta f}{\delta\bar{s}},\frac{\delta h}{\delta \omega}
\right\rangle
-\left\langle\frac{\delta h}{\delta\bar{s}}\otimes\bar{s},\eta\right\rangle
-\left\langle\frac{\delta h}{\delta\bar{s}},\xi\right\rangle, 
\end{align}
and $$\{f,h\}=\{f,h\}_{\mathrm{LP}}+ \{f,h\}_{E}.$$

Taking into account the splitting $\tilde{\mathfrak{g}}_{\mathrm{aff}}\simeq \tilde{\mathfrak{g}} \oplus E$, the expression of the divergence induced by a $G_{\mathrm{aff}}$-connection \eqref{eq: aff divergence}, the adjoint operator \eqref{eq: aff cov bracket}, and the expression of $\mathsf{P}^+$ in \eqref{eq: aff Ps adjoint}, the reduced equations of motion \eqref{eq: Lie-Poisson} particularize into the following set of equations:
\begin{equation} \label{eq: example affine g}
    \operatorname{div}^{\Lambda}\mu
- \operatorname{ad}^{*}_{\frac{\delta h}{\delta \mu}} \mu
+ \frac{\delta h}{\delta \omega} \otimes \omega
+ \frac{\delta h}{\delta \bar{s}} \otimes \bar{s}
= 0,
\end{equation}
\begin{equation} \label{eq: example affine e}
    \operatorname{div}^{\Lambda}\,\omega
+ \left(\frac{\delta h}{\delta \mu}\right)^*\!\cdot \omega
+ \frac{\delta h}{\delta \bar{s}} = 0,
\end{equation}
where $\left\langle \left(\frac{\delta h}{\delta \mu}\right)^*\!\cdot \omega, \xi\right\rangle =\left\langle \omega, \left(\frac{\delta h}{\delta \mu}\right)\!\cdot  \xi\right\rangle$ for every $\xi\in E$. Furthermore, Equation \eqref{eq: parallel} becomes: 
\begin{equation} \label{eq: example affine holonomy}
    \nabla^{\sigma}\bar{s}
+ \frac{\delta h}{\delta \omega} = 0,
\end{equation}
where $\sigma=(\mu,\omega,\bar{s})^*\left(\frac{\delta h}{\delta \mu}+\Lambda\right)$ is the $\g$ component of the connection
\begin{equation*}
    \sigma_{\mathrm{aff}}=(\mu,\omega,\bar{s})^*\left(\frac{\delta h}{\delta\mu_{\mathrm{afff}}}\right)+\Lambda=(\mu,\omega,\bar{s})^*\left(\frac{\delta h}{\delta \mu}+\Lambda,\frac{\delta h}{\delta \omega}\right)
\end{equation*}
From Theorem \eqref{th: reconstruction}, the reconstruction condition is $\Curv(\sigma_{\mathrm{aff}}) = 0$ which is equivalent to $\Curv(\sigma) = 0$ since from Equation \eqref{eq: example affine e},
$$
\operatorname{Curv}(\sigma_{\mathrm{aff}})
=
\left(
\operatorname{Curv}(\sigma),
\nabla^{\sigma} \left( \frac{\delta h}{\delta \omega}\right)\right)
=
\left(
\operatorname{Curv}(\sigma),
\operatorname{Curv}(\sigma)\wedge\bar{s}
\right).$$
Equations \eqref{eq: example affine g} and \eqref{eq: example affine e} are compatible with the Lagrangian counterpart \cite[Eq. 33]{EPsubgroup} except that these equations are 
expressed using the background connection $\Lambda$ used to define a Hamiltonian density $\mathcal{H}$ instead of using the connection $\sigma$. In turn, Equation \eqref{eq: example affine holonomy} with one of the compatibility equation \cite[Eq. 34]{EPsubgroup}. The other compatibily condition in \cite{EPsubgroup} is $\Curv(\sigma) = 0$ which in this context is rather interpreted as a reconstruction condition.


Given $(\mu,\omega,\bar{s})\in\Gamma(\Pi_{P_{\mathrm{aff}}} / G)$ a solution of the reduced system, consider $\bar{\mu}=\mu-\omega\otimes\bar{s}\in\Gamma(T M \otimes\g^*)$.
For all $\eta\in$
\begin{align*}
\left\langle \operatorname{div}^{\Lambda}(\omega \otimes \bar{s}), \eta \right\rangle
&=  \operatorname{div}^{\Lambda}(\left\langle\omega , \eta \cdot \bar{s}\right\rangle)  
= \left\langle \operatorname{div}^{\Lambda} \omega, \eta \cdot \bar{s} \right\rangle
  + \left\langle \omega, \nabla^{\Lambda}(\eta \cdot \bar{s}) \right\rangle \\
&= \left\langle \operatorname{div}^{\Lambda} \omega \otimes \bar{s}, \eta \ \right\rangle
  + \left\langle \omega, \eta \cdot \nabla^{\Lambda} \bar{s} \right\rangle
  = \left\langle \operatorname{div}^{\Lambda} \omega \otimes \bar{s}
  + \nabla^{\wedge} \bar{s} \otimes \omega, \eta \right\rangle
\end{align*}
and 
\begin{align*}
\left\langle \operatorname{ad}^*_{\frac{\delta h}{\delta \mu}}\omega \otimes \bar{s}, \eta \right\rangle
&= \left\langle \omega \otimes \bar{s},
  \left[ \frac{\delta h}{\delta \mu}, \eta \right] \right\rangle
= \left\langle \omega,
  \left[ \frac{\delta h}{\delta \mu}, \eta \right]\cdot\bar{s} \right\rangle
= - \left\langle \omega,
  \frac{\delta h}{\delta \bar{\mu}}(\eta \cdot \bar{s}) \right\rangle \\
&= - \left\langle \left( \frac{\delta h}{\delta \mu} \right)^{*} \omega,
  \eta \cdot \bar{s} \right\rangle 
= - \left\langle
  \left( \frac{\delta h}{\delta \mu} \right)^{*} \omega \otimes \bar{s},
  \eta \right\rangle
\end{align*}
Hence, using equations \eqref{eq: example affine g}, \eqref{eq: example affine e} and \eqref{eq: example affine holonomy}; 
\begin{align} \label{eq: affine momentum}
\operatorname{div}^{\Lambda} \bar{\mu}
- \operatorname{ad}^{*}_{\frac{\delta h}{\delta \mu}} \bar{\mu}
&= \operatorname{div}^{\Lambda} \mu
 - \operatorname{div}^{\Lambda} (\omega \otimes \bar{s})
 - \operatorname{ad}^{*}_{\frac{\delta h}{\delta \mu}} \mu
 + \operatorname{ad}^{*}_{\frac{\delta h}{\delta \mu}} \omega \otimes \bar{s} \nonumber\\
&= \operatorname{div}^{\Lambda} \mu
 - (\operatorname{div}^{\Lambda} \omega ) \otimes \bar{s}
 - \nabla^{\Lambda} \bar{s} \otimes \omega
 - \operatorname{ad}^{*}_{\frac{\delta h}{\delta \mu}} \mu 
 - \left( \frac{\delta h}{\delta \mu} \right)^*
   \omega \otimes \bar{s} \nonumber\\
&= \operatorname{div}^{\Lambda} \mu
 - (\operatorname{div}^{\Lambda} \omega ) \otimes \bar{s}
 - \frac{\delta h}{\delta \omega} \otimes \omega
 - \operatorname{ad}^{*}_{\frac{\delta h}{\delta \mu}} \mu
 - \left( \frac{\delta h}{\delta \mu} \right)^*
   \omega \otimes \bar{s} \nonumber\\
&= \operatorname{div}^{\Lambda} \mu
 - \operatorname{ad}^{*}_{\frac{\delta h}{\delta \mu}} \mu
 - \frac{\delta h}{\delta \omega} \otimes \omega
 + \frac{\delta h}{\delta \bar{s}} \otimes \bar{s} =0,
\end{align}
and we conclude that  $\bar{\mu}$ is a conserved quantity along the solutions of the reduced Hamiltonian system. Indeed, from \eqref{eq: affine momentum} we conclude that $\operatorname{div}^{\sigma} \bar{\mu}=0$. In accordance with the conservation law obtained in \cite{affinered} and  \cite{EPsubgroup}.

\subsection{A vierbein Einstein-Palatini framework}

We illustrate how the Hamiltonian reduction procedure developed in this paper provides a natural geometric framework for field theories on the frame bundle, with particular relevance for gravity. 


A vierbein is nothing but a (local) moving frame, that is, a (local) section of the frame bundle $P=LM$ over a manifold $M$. We consider a subgroup $K$ of the structure group $G=GL(n)$. 

On the Lagrangian side, suppose that one has a first-order $K$-invariant theory on $J^1LM$. The variables of the reduced problem are the sections $(\sigma , \bar{s})$ of the reduced phase bundle $(J^1LM)/K\simeq C\times (LM/K)$, that is, a linear connection on $M$ and a reduction of the frame bundle. In particular, for $K=SO(1,n-1)$, sections of $LM/K$ correspond to Lorentzian metrics on $M$. One reduced equation is $\nabla ^\sigma \bar{s}=0$ which means that the connection $\sigma$ is metric. This is shared to any choice of Lagrangian, in particular, a Palatini Lagrangian. On the other hand, the reconstruction condition $\Curv(\sigma)=0$,
turns out to be too restrictive. As discussed in \cite{Capriotti2014},\cite{Capriotti2025}, this obstruction can be avoided by modifying the variational principle, replacing the contact ideal on $J^1P$ with an exterior differential system that encodes metricity and torsion constraints while allowing nontrivial curvature.

On the Hamiltonian side, let $\Pi_{LM}$ be the polysymplectic bundle and let $\mathcal{H}=H\vol$ be a $K$-invariant Hamiltonian density. From Proposition \ref{prop: map}, the reduced variables are sections
\[
(\mu,\bar{s}):M \to \Pi_{LM}/K \cong \left(TM\otimes \mathrm{End}^*(TM)\otimes \bigwedge\nolimits^{n-1}T^*M\right)\times_M P/K,
\]
where $\mu$ may be interpreted as polymomenta dual to a connection on $M$ and $\bar{s}$ is a Lorentzian metric on $M$. The reduced equations are given by Theorem~\ref{th: reduced dynamics}. Equation \eqref{eq: Lie-Poisson} yields the dynamical field equations, which are analogous to Einstein-type equations in this setting, and equation \eqref{eq: parallel} expresses again metric compatibility (see \cite{Capriotti2025}).

A key difference with the Lagrangian case is that no flatness condition appears in the reduced dynamics. By Theorem~\ref{th: reconstruction}, flatness of the induced connection
\[
\sigma = (\mu,\bar{s})^*\!\left(\frac{\partial h}{\partial \mu}\right) + \Lambda
\]
is required only for reconstruction of an unreduced solution. Thus, the reduced Hamiltonian equations describe a metric-affine theory with compatibility condition but unconstrained curvature.

We do not attempt here to construct a Hamiltonian formulation fully equivalent to Palatini gravity, but rather to highlight how the geometric structure of the reduced equations naturally accommodates metric-affine theories. To recover Palatini theory, one must further impose torsion constraints. Due to the singular nature of the Palatini Lagrangian, this typically requires restricting the phase space, as in the multisymplectic formulation of \cite{Gaset2019}. From the Hamiltonian point of view, this corresponds to working on a suitable constraint submanifold.

This example shows that, unlike in the Lagrangian reduction, where flatness is built into the variational structure, in the Hamiltonian framework flatness appears only as a reconstruction condition, allowing for a natural interpretation of the reduced equations as a Palatini-type theory.

\bibliography{references}{}
\bibliographystyle{abbrvbf}
\end{document}